\documentclass[reqno,11pt]{article} 
\usepackage{amsthm,amsfonts,amssymb,euscript,mathrsfs,graphics,color,amsmath,latexsym,hyperref}
\usepackage[latin1]{inputenc}

\usepackage[english]{babel}

\theoremstyle{plain}

\usepackage{times}
\numberwithin{equation}{section}

\newtheorem{theorem}{Theorem}[section]

\newtheorem{proposition}[theorem]{Proposition}
\newtheorem{lemma}[theorem]{Lemma}

\newtheorem{remark}[theorem]{Remark}
\newtheorem{remarks}[theorem]{Remark}
\newtheorem{definition}[theorem]{Definition}
\newcommand{\be}{\begin{equation}}
\newcommand{\ee}{\end{equation}}
\newcommand{\teta}{\theta}

\newcommand{\e}{\varepsilon}
\newcommand{\ep}{\epsilon}
\newcommand{\al}{\alpha}

\newcommand{\ov}{\overline}

\newcommand{\R}{\mathbb R}
\newcommand{\C}{\mathbb C}

\newcommand{\Z}{\mathbb Z}

\newcommand{\N}{\mathbb N}
\newcommand{\T}{\mathbb T}

\newcommand{\s }{\sigma }
\newcommand{\ii }{{\rm i} }

\newcommand{\bral}{[ \! [} 
\newcommand{\brar}{] \! ]}

\newcommand{\dps}{\displaystyle}

\newcommand{\x }{\xi }

\newcommand{\pa}{\partial}

\newcommand{\opbw}{{Op^{\mathrm{BW}}}}

\def\hat{\widehat}

\def\bar{\overline}

\def\cal{\mathcal}

\renewcommand{\Im}{\mathrm{Im}\,}

\newcommand{\mk}{{\Omega}}

\def\ba{\begin{aligned}}
\def\ea{\end{aligned}}
\def\beginm{\begin{multline}}
\def\endm{\end{multline}}

\providecommand{\vect}[2]{{\bigl[\begin{smallmatrix}#1\\#2\end{smallmatrix}\bigr]}}   
\providecommand{\sm}[4]{{\bigl[\begin{smallmatrix}#1&#2\\#3&#4\end{smallmatrix}\bigr]}}

\begin{document}

\title{\bf 
Quadratic life span
 \\  of periodic gravity-capillary water waves} 
\date{}

\author{M. Berti, R. Feola, L. Franzoi}



\maketitle

\begin{abstract}
We consider the gravity-capillary water waves equations for a bi-dimensional fluid 
with a  periodic one-dimensional free surface. We prove a rigorous reduction of this system to Birkhoff normal form up to cubic degree.  Due to the possible presence of
$ 3 $-waves resonances for general values of 
gravity, surface tension  and depth,
such normal form may be not trivial and exhibit  
a chaotic dynamics 
(Wilton-ripples). Nevertheless we prove that 
for {\it all} the values of gravity, surface tension and depth,  
initial data that are of size $ \e $ in a sufficiently smooth Sobolev space lead to a 
solution that  
remains in an $ \e $-ball of the same Sobolev space up times of order $ \e^{-2}$.
We exploit that 
the $ 3$-waves resonances 
are 
finitely many, and  the Hamiltonian nature of the Birkhoff normal form.
\end{abstract}


\section{Introduction and main results} 

We consider an incompressible and irrotational perfect fluid, under the action of gravity,
occupying at time
$t$ the bi-dimensional domain
\[
{\mathcal D}_{\eta} := \big\{ (x,y)\in \T \times\R \, ;  \ - h <y<\eta(t,x) \big\}\,,  
\quad  \T := \R \slash (2 \pi \Z) \, ,
\]
periodic in the horizontal variable, with  depth $ h $ which may be finite or infinite. 
The time-evolution of the fluid is determined by a system of equations for 
the free surface $ \eta(t,x)$ 
and the function $ \psi(t,x) := \Phi(t,x,\eta(t,x))$
 where $\Phi$ is the velocity potential in the fluid domain.
Given the shape $\eta(t,x)$ of the domain $ {\cal D}_\eta $ 
and the Dirichlet value $\psi(t,x)$ of the velocity potential at the top boundary, one 
recovers 
$\Phi(t,x,y)$ as the unique solution of the elliptic problem 
$$
\Delta \Phi = 0  \ \text{in } 
{\cal D}_\eta  \, , \quad 
\partial_y \Phi = 0  \ \text{at } y = - h  \, , \quad 
\Phi = \psi \  \text{on } \ \{y = \eta(t,x)\} \, . 
$$
According to Zakharov \cite{Zak1} and Craig-Sulem \cite{CrSu}, the
$(\eta,\psi)$ variables evolve under the system
\begin{equation} \label{WW1}
\left\{\begin{aligned}
 &   \partial_t \eta = G(\eta)\psi \\
&\partial_t\psi = \dps -g\eta  -\frac{1}{2} \psi_x^2 
+  \frac{1}{2}\frac{(\eta_x  \psi_x + G(\eta)\psi)^2}{1+\eta_x^2}+\kappa 
\partial_x\Big( \frac{\eta_x}{(1+\eta_x^{2})^{\frac{1}{2}}}\Big) 
\end{aligned}\right.
\end{equation} 
where $g>0$ is the acceleration of gravity, $\kappa>0$ the surface tension,
and $G(\eta)$ is 
 the nonlocal Dirichlet-Neumann operator
defined by 
$ G(\eta)\psi := (\partial_y\Phi -\eta_x \partial_x\Phi)(t,x,\eta(t,x)) $.

As  observed by Zakharov  \cite{Zak1}, 
the equations \eqref{WW1} are the Hamiltonian system  
$$
 \pa_t \eta = \nabla_\psi H (\eta, \psi)\,,
 \qquad \pa_t \psi = - \nabla_\eta H (\eta, \psi)\,,
$$
where $ \nabla $ denotes the $ L^2 $-gradient, with Hamiltonian
\be\label{Hamiltonian}
H(\eta, \psi) := \frac12 \int_\T \psi \, G(\eta ) \psi \, dx + \frac{g}{2} \int_{\T} \eta^2  \, dx+
\kappa\int_{\mathbb{T}} \sqrt{1 + \eta_x^2  }\,dx
\ee
given by the sum of the kinetic and potential energy of the fluid and the 
energy of the capillary forces. We remind that the Poisson bracket between
two functions $H(\eta,\psi)$, $F(\eta,\psi)$ is
\begin{equation}\label{PoiBra}
\{H,F\}=\int_{\mathbb{T}}(\nabla_{\eta}H \nabla_{\psi}F
-\nabla_{\psi}H \nabla_{\eta}F) dx\,.
\end{equation}
The ``mass'' $ \int_\T \eta \, dx$ is a prime integral of \eqref{WW1} and, 
with no loss of generality, we can fix it to zero by shifting the $y$ coordinate. 
Moreover \eqref{WW1} is invariant under spatial translations
and Noether's theorem implies that the momentum $ \int_{\T} \eta_x (x) \psi (x)  \, dx $
is a prime integral of \eqref{WW1}.

Let $ H^s (\T) := H^s $, $ s\in \R $, denote the Sobolev spaces of $ 2 \pi $-periodic functions of $ x $. The variable $ \eta $ belongs
to the subspace  $H^s_0(\T)$  of $H^s(\T)$ of zero average functions (for some 
positive $ s  $). 
On the other hand, the variable $ \psi $ belongs to  the homogeneous Sobolev space
${\dot H}^s (\T):= H^s (\T) \slash_{ \sim}$ 
obtained by the equivalence relation
$\psi_1 (x) \sim \psi_2 (x)$ if and only if $ \psi_1 (x) - \psi_2 (x) = c $ is a constant.
This is coherent 
with the fact that only the velocity field $ \nabla_{x,y} \Phi $ has physical meaning, 
and the velocity potential $ \Phi $ is defined up to a constant.
For simplicity
we denote the equivalence class $ [\psi] $ in $\dot{H}^s$ by $ \psi $ and,  
since the quotient map induces an isometry of $ {\dot H}^s (\T) $ onto $ H^s_0 (\T) $, 
we conveniently identify $ \psi $ with a function with zero average.

The water waves equations \eqref{WW1}  are a quasi-linear system.   
In the last years, they have been object of intense research 
both in the periodic setting $ x \in \T^{d} $, $ d = 1, 2 $, 
and in the dispersive 
case $ x \in \R^{d} $ 
with data decaying at infinity.
A fundamental  difference between these cases concerns the 
dynamical behavior of the 
 linearized water waves equations at $ (\eta, \psi) $= 0. 
In $d=1$ they  are
\be\label{LINWW}
\left\{\begin{aligned}
 &   \partial_t \eta = G(0) \psi  \\
&\partial_t\psi = \dps -g\eta + \kappa \eta_{xx} \, , 
\end{aligned}\right.  \qquad G(0) = \tanh (h D) D \, , \quad D := \frac{1}{\ii} \pa_x \, ,  
\ee
with dispersion relation 
 \begin{equation}\label{dispersionLaw}
\mk(\xi):=\mk_{g,\kappa,h}(\xi):=
( \kappa |\xi|^{3} + g|\xi | )^{\frac{1}{2}} \big(  \tanh(h |\xi| )\big)^{\frac{1}{2}} \, . 
\end{equation}
Notice that, if $ h = + \infty $, the Dirichlet-Neumann operator is $ G(0) = |D | $ and  
 the dispersion relation $ \mk(\xi) =  ( \kappa |\xi|^{3} + g|\xi | )^{\frac{1}{2}}  $. 
In the case $ x \in \R^d $, the solutions of \eqref{LINWW} disperse to zero as $ t \to + \infty $.
On the contrary, 
if $ x \in \T^d $, all the solutions 
of the linear system \eqref{LINWW} are time periodic, or quasi-periodic, 
or almost periodic in time, 
with linear frequencies of oscillations $ \mk (j) $, $ j \in \Z^d $. 
In such a case a natural tool to analyze the nonlinear dynamics of 
\eqref{WW1}, at least for small amplitude solutions,  is normal form theory, that is particularly difficult 
due to the quasi-linear nature of the nonlinearity. 
In \cite{CS}, Craig and Sulem developed 
a Birkhoff normal form analysis for \eqref{WW1} 
starting from the Taylor expansion of the Hamiltonian \eqref{Hamiltonian}, 
$$
H =H^{(2)}+H^{(3)} +H^{(\geq4)} \, , 
$$
where (up to a constant)
\begin{equation}\label{Hamvera}
\begin{aligned}
& H^{(2)} := \frac12 \int_{\mathbb{T}}\psi G(0) \psi\,dx+ 
\frac{g}{2} \int_{\mathbb{T}}\eta^2 + 
 \frac{\kappa}{2} \int_{\T} \eta_x^2 \, dx \\
& H^{(3)} := \frac{1}{2} \int_{\T} \psi \Big( D \eta D - G(0) \eta G(0) \Big) \psi \, dx  
\end{aligned}
\end{equation}
and 
$H^{(\geq4)}$ collects all the terms of homogeneity in $ (\eta,\psi) $
greater or equal than $4$.
Unfortunately, in this Taylor expansion 
there is a priori no
control of the unboundedness of the Hamiltonian vector field associated to 
$ H^{(\geq 4)} $.

Normal form theory for gravity-capillary water waves,  
even in $ x $, has been developed in Berti-Delort 
\cite{BD}, proving, for most values of the parameters
$ (g, \kappa) $,   
an almost global  existence result for the solutions of \eqref{WW1}
 in Sobolev spaces.   
 A key point is, in analogy with the KAM theory approach in \cite{BM}, 
to transform the unbounded water waves 
vector field to a paradifferential one with constant coefficient symbols, up to smoothing 
operators.  
Very recently, Birkhoff normal form and long time existence results 
for periodic pure gravity water waves
in infinite depth, where
no parameters are available, have been proved  in Berti-Feola-Pusateri \cite{BFP}.
A key point is a normal form uniqueness 
argument which allows to identify 
the paradifferential normal form with the formal 
Hamiltonian 
 Birkhoff normal form up to fourth degree, which turns out to be completely integrable.  

Complementing these works,  the goal of this paper 
is to prove  that, for {\it any} value 
of $ (\kappa, g, h) $, $ \kappa > 0 $,  the gravity-capillary water waves system
 \eqref{WW1} is conjugated to its
Birkhoff normal form, up to cubic remainders that satisfy energy estimates 
(Theorem \ref{BNFtheorem}), 
and that all the solutions of \eqref{WW1}, 
with initial data  of size $ \e $ in a sufficiently smooth Sobolev space, exist
and remain in an $ \e $-ball of the same Sobolev space up times of order $ \e^{-2}$, see Theorem \ref{thm:main2}. Let us state precisely these results. 
\\[1mm]
{\it Main results.}
To state our first main result,  
concerning the rigorous reduction of  system \eqref{WW1} 
to its Birkhoff normal form up to cubic degree, 
 let us assume that, for $s$ large enough and some $T>0$, 
we have a classical solution
\begin{align}\label{etapsi}
(\eta,\psi)\in C^{0}([-T,T]; H^{s+\frac{1}{4}}_0 \times {\dot H}^{s-\frac{1}{4}})
\end{align}
of the Cauchy problem for \eqref{WW1}.
The existence of such a solution, at least for small enough $ T $, 
is guaranteed by local well-posedness theory, 
see the literature at the end of the section.

\begin{theorem}{\bf (Cubic Birkhoff normal form)} 
\label{BNFtheorem}
Let $\kappa> 0$, $g\geq 0$ and $h \in (0, + \infty]$.
There exist $ s \gg 
1 $  and $ 0 < \bar{\e} \ll 1  $,   such that, if 
$(\eta,\psi) $ is a solution of \eqref{WW1} satisfying \eqref{etapsi} with  
 \begin{equation}\label{hypoeta}
 \sup_{t\in [-T,T]}\big( 
 \|\eta\|_{H_0^{s+\frac{1}{4}}}+\|\psi\|_{\dot{H}^{s-\frac{1}{4}}}\big)
 \leq \bar \e\,,
 \end{equation}
then there exists a bounded 
and invertible linear operator 
$ \mathfrak{B}(\eta, \psi): 
H^{s+\frac{1}{4}}_0 \times {\dot H}^{s-\frac{1}{4}} \to  \dot H^s $,  
which depends (nonlinearly) on $ (\eta, \psi) $, 
such that
\begin{equation}\label{Germe}
\begin{aligned}
 {\|\mathfrak{B}(\eta, \psi) \|}_{\mathcal{L}(H^{s+\frac{1}{4}}_0 \times {\dot H}^{s-\frac{1}{4}}, \dot{H}^{s})}
	+ {\|(\mathfrak{B}(\eta, \psi))^{-1} \|}_{\mathcal{L}(\dot{H}^{s}, 
	H^{s+\frac{1}{4}}_0 \times {\dot H}^{s-\frac{1}{4}})} \leq & \\ 
 1+ 
	C(s) (\|\eta\|_{H_0^{s+\frac{1}{4}}}+\|\psi\|_{\dot{H}^{s-\frac{1}{4}}}) \, , &  
	\end{aligned}
\end{equation}
and the variable 
$z:=\mathfrak{B}(\eta, \psi)[\eta, \psi]$
satisfies the equation
\begin{equation}\label{theoBireq}
\pa_{t}z =\ii \mk(D)z +\ii \pa_{\bar{z}}H_{BNF}^{(3)}(z,\bar{z}) 
+ {\mathcal X}^{+}_{\geq 3}
\end{equation}
where:  

\vspace{0.5em}
\noindent
$(0)$ $\mk(D) $ is the Fourier multiplier with symbol defined in \eqref{dispersionLaw} 
and  
$ \pa_{\bar{z}} $ is defined in \eqref{HamVecField},

\noindent 
$(1)$ the Hamiltonian  $H^{(3)}_{BNF}(z, \bar z )$  has the form
\begin{equation}\label{H3bnf}
\begin{aligned}
H^{(3)}_{BNF}(z, \bar z)& =
\sum_{\substack{\s_1j_1+\s_2j_2+\s_2j_3=0, \, \s_i = \pm \, , \\
\s_1\mk(j_1)+\s_2\mk(j_2)+\s_3\mk(j_3)=0, j_i \in \mathbb{Z}\setminus\{0\} } } \!\!\!\!\!
H_{j_1,j_2,j_3}^{\s_1,\s_2,\s_3} z_{j_1}^{\s_1}z_{j_2}^{\s_2}z_{j_3}^{\s_3}
\end{aligned}
\end{equation}
where $ z_j^+ := z_j $, $ z_j^{-} := \ov{z_j} $
and  $z_j$ denotes the $ j  $-th Fourier coefficient of the function $ z $
(see \eqref{complex-uU}),  
and the coefficients
\begin{equation}\label{coeffH3}
H_{j_1,j_2,j_3}^{\s_1,\s_2,\s_3}:=\frac{\ii \s_2 }{8\sqrt{\pi}}
\big(\s_1\s_3j_1j_3+G_{j_1}G_{j_3}\big)
\frac{\Lambda(j_2)}{\Lambda(j_1)\Lambda(j_3)}
\end{equation}
with $\Lambda (j) $ defined in \eqref{compl1} and 
$G_{j}:=\tanh(hj)j$;

\vspace{0.5em}
\noindent
$(2)$ $ {\mathcal X}^{+}_{\geq 3}  := {\mathcal X}^{+}_{\geq 3} (\eta,\psi,z,\bar{z})$ 
satisfies 
$ \| {\mathcal X}^{+}_{\geq 3} \|_{{\dot H}^{s-\frac32}} \leq C(s) \| z \|_{\dot{H}^s}^3 $
and the ``energy estimate''
\begin{equation}\label{theoBirR}
{\rm Re}\int_{\T}|D|^s {\mathcal X}^{+}_{\geq 3} \cdot \bar{|D|^s z} \, dx
\leq C(s)  \|z\|_{\dot{H}^s}^{4}  \, .
\end{equation} 
\end{theorem}

The main point of Theorem \ref{BNFtheorem}  is the construction of the bounded and invertible transformation $ \mathfrak{B}(\eta, \psi)  $ in \eqref{Germe} which recasts 
\eqref{WW1} in the Birkhoff normal form \eqref{theoBireq}, where the cubic vector field 
satisfies the energy estimate \eqref{theoBirR}. 
We remark that   Craig-Sulem \cite{CS}   constructed 
a bounded and symplectic transformation 
that conjugates \eqref{WW1}
to its cubic Birkhoff normal form, but the cubic terms 
 of the transformed vector field do not satisfy  energy estimates.
  
 We underline that,
for general values of gravity, surface tension and depth 
$ (g, \kappa, h)$,
the ``resonant" Birkhoff normal form Hamiltonian 
 $H^{(3)}_{BNF}$ in \eqref{H3bnf} is non zero, because 
the system 
\be\label{3-w-res}
\s_1 \mk({j_1})+\s_2 \mk({j_2})+\s_3 \mk({j_3})=0 \, , \qquad 
\s_1j_1+\s_2j_2+\s_3j_3=0 \, , 
\ee
for $\s_{j} = \pm $ , 
may possess integer  solutions 
$j_1,j_2,j_3\neq 0 $, known as $3$-waves resonances
(cases with absence of $3$-waves resonances are discussed in 
remark \ref{rem:no-reso}). 
The  resonant Hamiltonian $H^{(3)}_{BNF} $
gives rise to a complicated dynamics, 
which,  in fluid mechanics,  is responsible for the phenomenon of the 
Wilton ripples.
Nevertheless we are able to prove the following 
long time stability result.

\begin{theorem}{\bf (Quadratic life span)}\label{thm:main2}
For any  value of $(\kappa, g, h) $, $ \kappa > 0 $, $ g \geq 0 $, $ h \in (0, + \infty] $, 
there exists $ s_0 > 0 $ and,  
for all $ s \geq s_0 $, there are $ \e_0 >  0 $, $ c > 0 $, $ C > 0 $, such that,  
for any $0< \e\leq \e_0$,
 any initial data 
\begin{equation}\label{smallness}
 (\eta_0, \psi_0)\in H_0^{s+\frac{1}{4}}(\mathbb{T},\mathbb{R})\times
  \dot{H}^{s-\frac{1}{4}}(\mathbb{T},\mathbb{R}) \quad { with}\quad
  \|\eta_0\|_{H_0^{s+\frac{1}{4}}}+\|\psi_0\|_{\dot{H}^{s-\frac{1}{4}}}\leq \e \, ,
\end{equation} 
there exists a unique classical solution $(\eta,\psi)$
of \eqref{WW1} belonging to 
\[
C^{0}\Big([-T_\e,T_{\e}], H_0^{s+\frac{1}{4}}(\mathbb{T},\mathbb{R})\times
  \dot{H}^{s-\frac{1}{4}}(\mathbb{T},\mathbb{R})
\Big)\, \quad { with} \quad T_{\e}\geq c\e^{-2}\,,
\]
satisfying $(\eta,\psi)_{|_{t=0}}=(\eta_0,\psi_0)$. Moreover
\begin{equation}\label{smallness2}
\sup_{t\in [-T_{\e},T_{\e}]}\big( 
 \|\eta\|_{H_0^{s+\frac{1}{4}}}+\|\psi\|_{\dot{H}^{s-\frac{1}{4}}}\big)\leq C\e\,.
\end{equation}
\end{theorem}

Before presenting 
the literature about  $ \e^{-2} $  
existence results for water waves, 
we 
describe some key points concerning the proof of these results: 

\smallskip
1.  The long time existence  Theorem \ref{thm:main2} is deduced 
by the complete conjugation of the water waves vector field 
\eqref{WW1}
to its Birkhoff normal form up to degree $3$, 
Theorem \ref{BNFtheorem}, 
and not just on the construction of modified energies. 

2. Since the gravity-capillary dispersion relation 
$\sim |\xi |^{\frac{3}{2}}$ is superlinear, 
the  water waves equations \eqref{WW1} can be  reduced, as in \cite{BD},
to a paradifferential system with constant coefficient symbols, 
up to smoothing remainders (see Proposition \ref{regolo}).
At the beginning of Section \ref{sec:5} we remark that, thanks to
the $ x $-translation invariance of the equations, 
the symbols in \eqref{finalsyst} of the quadratic paradifferential vector fields  are actually zero. 
For this reason, in Section \ref{sec:5}, it just remains 
to perform a  Poincar\'e- Birkhoff normal form
on the quadratic smoothing vector fields, see Proposition \ref{cor:BNF10}.

3. Despite the fact that
our transformations are non-symplectic (as in \cite{BD}, \cite{BFP}),
 we prove, in Section \ref{sec:IDE}, 
 using a normal form identification argument 
 (simpler than  in \cite{BFP}),
 that the quadratic 
  Poincar\'e-Birkhoff normal form term in \eqref{finalsyst1012}
 coincides with the Hamiltonian vector field $ \ii \pa_{\ov{z}} H^{(3)}_{BNF} $
with  Hamiltonian \eqref{H3bnf}.

4. The Hamiltonian 
$ H^{(2)}_\C (z) :=  \int_{\T} \Omega (D) z  \cdot \overline{z} \, dx $ is a prime integral  
of the resonant  Birkhoff normal form
$ \pa_{t}z =\ii \mk(D)z +\ii \pa_{\bar{z}}H_{BNF}^{(3)}(z, \bar z) $.
Moreover, since \eqref{3-w-res} admits at most finitely many integer solutions
(Lemma \ref{stimedalbasso}) the Hamiltonian $ H^{(3)}_{BNF} (z, \bar z) 
= H^{(3)}_{BNF} (z_L, \bar z_L ) $
where $ z_L := \sum_{0 < |j| \leq \mathtt C} z_j e^{\ii j x } $, for some finite ${\mathtt C} > 0  $.
Therefore, 
 any solution $ z (t) $ of the Birkhoff normal form satisfies, for any $ s \geq 0 $, 
$$
\| z_L (t) \|_{\dot H^s}^2  \lesssim_s \| z_L (t) \|_{L^2}^2 \lesssim 
H^{(2)}_\C (z_L (t)) = H^{(2)}_\C (z_L (0) ) 
 \, , 
\  \  \forall t \in \R \, , 
$$
and $ \| z (t) \|_{\dot H^s}^2 $ remains bounded for all times. 
Finally  we deduce 
 the energy estimate \eqref{stimaconclu} for the solution  of 
 the whole system \eqref{theoBireq}, where we take into account the effect of 
 $ {\cal X}_{\geq 3 }^+ $, 
which implies stability for all $ |t| \leq c \e^{-2}$. 

\smallskip
\noindent 
{\it Literature.}
Local existence results for the initial value problem of the water waves equations go
back to the pioneering works of
Nalimov \cite{Nali}, Yosihara \cite{Yosi}, Craig \cite{Crai} for small initial data, Wu \cite{Wu97,Wu99} without smallness  assumptions, and Beyer-G\"unther \cite{BeGun} in presence of surface tension. 
For some recent results about gravity-capillary waves we refer to
 \cite{Ambro1, LannesLivre, Cou, ShaZ, ChrSta, ABZduke}. 
Clearly, specializing these results
 for initial data of size $ \e $, the solutions exist and stay regular for times
of order $ \e^{-1} $.
\\[1mm]
{\it Global well-posedness.} 
In the case $ x \in \R^d $ 
and the initial data decay sufficiently fast at infinity,  
global in time solutions  have been constructed 
exploiting the dispersive effects of the system. 
The first global in time solutions were proved in $ d = 2 $
 by  Germain-Masmoudi-Shatah \cite{GMS} and Wu
\cite{Wu2} for gravity water waves, 
by Germain-Masmoudi-Shatah \cite{GMS2}
for the capillary problem, and  
for  gravity-capillary water waves 
 by Deng-Ionescu-Pausader-Pusateri \cite{DIPP}. 
In $ d = 1 $ an almost global existence result for  gravity waves
was proved by Wu \cite{Wu},   
improved to global regularity by 
 Ionescu-Pusateri \cite{IP}, Alazard-Delort \cite{AlDe1},  
 Hunter and Ifrim-Tataru \cite{HuIT, IFRT}. For  capillary waves, 
global regularity was proved by  Ionescu-Pusateri \cite{IP3}
 and Ifrim-Tataru \cite{ifrTat}.
\\[1mm]
{\it Normal forms.} 
For space periodic water waves 
in absence of   $3$-waves resonances, 
existence results for times of order $ \e^{-2} $ 
have been obtained
in \cite{Wu,ToWu,IP,AlDe1,HuIT} 
for $ 1d$ pure gravity waves,  in \cite{IP3,ifrTat} for pure capillarity waves, 
and in \cite{HarIT} for $ 1d$  gravity waves over a flat bottom.
If $ x \in \T^2 $ we refer to \cite{IP2018} for an $ \e^{-\frac53 +}$ result.   
The only  $ \e^{-3} $  existence result for parameter independent water waves
is proved in \cite{BFP}, and it is based on the complete integrability of the fourth order 
Birkhoff normal form for $ 1d$ pure gravity  water waves in infinite depth. 

An almost global  existence
result 
of periodic gravity-capillary  water waves,  even in $ x $, 
for times $ O(\e^{-N})$ 
has been proved by Berti-Delort \cite{BD},
for almost all values of $ (g, \kappa) $.  
The restriction on the parameters $ (g, \kappa) $ arises to 
verify the absence of $ N $-waves interactions at any $ N $. The restriction 
to even in $ x $ solutions arises because the transformations in \cite{BD} are 
reversibility preserving but not symplectic. 
Almost global existence results for 
 fully nonlinear reversible Schr\"odinger equations have been proved in 
 \cite{FI}. 

We finally mention that time 
quasi-periodic solutions for \eqref{WW1} 
have been constructed in Berti-Montalto
\cite{BM} and, for pure gravity waves, in Baldi-Berti-Haus-Montalto \cite{BBHM}.

\smallskip
\noindent
{\it Acknowledgements}. 
The research was 
 partially 
supported by PRIN 2015 KB9WPT-005  
and ERC project FAnFArE,   n. 637510.

 \section{Functional Setting and Paradifferential calculus}
 
 In this section we  recall  definitions and results  
 of para-differential calculus 
 following Chapter $3$ 
 of  \cite{BD},  where we refer for more information. 
In the sequel we will deal with parameters
$$
s \geq s_0 \gg K \gg \rho \gg 1 \, . 
$$
Given  an interval $I \subset \R$, symmetric with respect to $t=0$, and  $s\in \R $, we define the space
$ C^K_{*}(I,{\dot{H}}^{s}(\T,\C^2)):=\bigcap_{k=0}^{K}C^{k}\big(I;\dot{H}^{s-\frac{3}{2}k}(\T;\C^2)\big)   $ 
endowed 
with the norm
$$
\sup_{t\in I}\|{U(t,\cdot)}\|_{K,s} \quad \mbox {where} 
\quad \|{U(t,\cdot)}\|_{K,s}:=\sum_{k=0}^{K}\|{\partial_t^k U(t,\cdot)}\|_{{\dot{H}}^{s-\frac{3}{2}k}}.
$$
With similar meaning we consider $C_{*}^{K}(I;\dot{H}^{s}(\T;\C))$.
We denote by $C^K_{*\R}(I,{\dot{H}}^{s}(\T,\C^2))$
the subspace of functions $U$ in $C^K_{*}(I,{\dot{H}}^{s}(\T,\C^2))$ such that $U=\vect{u}{\bar{u}}$.
Given $r > 0 $ we set
\begin{equation}\label{palla}
B_{s}^K(I;r):=\Big\{U\in C^K_{*}(I,\dot{H}^{s}(\T;\C^{2})):\, \sup_{t\in I}\|{U(t,\cdot)}\|_{K,s}<r \Big\} \, .
\end{equation}
 We expand a $ 2 \pi $-periodic function $ u(x) $, with zero average in $ x $, 
(which is identified with  $ u $ in the homogeneous space),  
 in Fourier series as 
\be\label{complex-uU}
u(x) = \sum_{n \in \Z\setminus\{0\} } \hat{u}(n)\frac{e^{\ii n x }}{\sqrt{2\pi}} \, , \qquad 
\hat{u}(n) := \frac{1}{\sqrt{2\pi}} \int_\T u(x) e^{-\ii n x } \, dx \, .
\ee
We also use the notation
$ u_n^+ := u_n := \hat{u}(n) $ and 
$  u_n^- := \ov{u_n}  := \ov{\hat{u}(n)} $. We set $ u^+ (x) := u(x) $ and 
$ u^{-} (x) := \ov{u(x)}$. 

For $n\in \N^*:= \N \! \smallsetminus \! \{0\}$ we denote by $\Pi_{n}$ the orthogonal projector from $L^{2}(\T;\C)$ 
to the subspace spanned by $\{e^{\ii n x}, e^{-\ii nx}\}$,  i.e.
$ (\Pi_{n}u)(x) := \hat{u}({n}) \frac{e^{\ii nx}}{\sqrt{2\pi}}+\hat{u}({-n})\frac{e^{-\ii nx}}{\sqrt{2\pi}} \, , $
and we denote by $ \Pi_n $ also the corresponding projector in  $L^{2}(\T,\C^{2})$.
If $\mathcal{U}=(U_1,\ldots,U_{p})$
is a $p$-tuple
 of functions, $\vec{n}=(n_1,\ldots,n_p)\in (\N^{*})^{p}$, we set
 $ \Pi_{\vec{n}}\mathcal{U}:=(\Pi_{n_1}U_1,\ldots,\Pi_{n_p}U_p) $. 
 
We deal with 
 vector fields $ X $  
which satisfy 
the {\it $x$-translation invariance} property
$$ 
X \circ \tau_\theta = \tau_\theta \circ X   \, , \quad  \forall\, \theta \in \R \, ,
\quad {\rm where}\quad \tau_\theta : u(x) \mapsto (\tau_\theta u)(x) := u(x + \theta )\,.
$$

\noindent{\bf Para-differential operators.}
We first give the definition of the classes of symbols,  
collecting Definitions $ 3.1$, $ 3.2$ and $ 3.4$ in \cite{BD}.
Roughly speaking, the class $\widetilde{\Gamma}_{p}^{m}$ contains homogeneous symbols of order $m$ and homogeneity $p$ in $U$, while the class $\Gamma_{K,K',p}^{m}$ contains non-homogeneous symbols of order $m$ which
vanish at degree at least $p$ in $U$, and that are $(K-K')$-times differentiable in $t$.

\begin{definition}{\bf (Classes of  symbols)}\label{pomosimb}
Let $m\in\R$, $p, N \in \N$ with $ p \leq N $,  $ K, K' $ in $\N$ with $K'\leq K$, $r>0 $.

\noindent
$(i)$ {\bf  $p$-homogeneous symbols.} We denote by $\widetilde{\Gamma}_{p}^{m}\!$ the space of symmetric $p$-linear maps
from $(\dot{H}^{\infty}(\T;\C^{2}))^{p}$ to the space of $C^{\infty}$ functions of $(x,\x)\in \T\times\R$, 
$ \mathcal{U}\to ((x,\x)\to a(\mathcal{U};x,\x)) $,  
satisfying the following. There is $\mu>0$ and,  
for any $\alpha,\beta\in \N $,  there is $C>0$ such that
\begin{equation}\label{pomosimbo1}
|\pa_{x}^{\alpha}\pa_{\x}^{\beta}a(\Pi_{\vec{n}}\mathcal{U};x,\x)|\leq C | \vec{n} |^{\mu+\alpha}\langle\x\rangle^{m-\beta}
\prod_{j=1}^{p}\|\Pi_{n_j}U_{j}\|_{L^{2}} 
\end{equation}
for any $\mathcal{U}=(U_1,\ldots, U_p)$ in $(\dot{H}^{\infty}(\T;\C^{2}))^{p}$,
and $\vec{n}=(n_1,\ldots,n_p)\in (\N^*)^{p}$. 
Moreover we assume that, if for some $(n_0,\ldots,n_{p})\in \N\times(\N^*)^{p}$,
$ \Pi_{n_0}a(\Pi_{n_1}U_1,\ldots, \Pi_{n_p}U_{p};\cdot)\neq0 $, 
then there exists a choice of signs $\s_0,\ldots,\s_p\in\{-1,1\}$ such that $\sum_{j=0}^{p}\s_j n_j=0$.
For $p=0$ we denote by $\widetilde{\Gamma}_{0}^{m}$ the space of constant coefficients symbols
$\x\mapsto a(\x)$ which satisfy \eqref{pomosimbo1} with $\al=0$
and the right hand side replaced by $C\langle \x\rangle^{m-\beta}$. 
In addition we require the translation invariance property
\be\label{def:tr-in}
a( \tau_\teta {\cal U}; x, \xi) =  a( {\cal U}; x + \theta, \xi) \, , \quad \forall \theta \in \R \, . 
\ee

\noindent
$(ii)$ {\bf Non-homogeneous symbols.} Let $ p \geq 1 $. We denote by $\Gamma^m_{K,K',p}[r]$ the space 
of functions $(U;t,x,\xi)\!\mapsto \!a(U;t,x,\xi)$, defined for $U\in B_{s_0}^K(I;r)$, for some large enough $ s_0$, with complex values such that for any 
$0\leq k\leq K-K'$, any $\s\geq s_0$, there are $C>0$, $0<r(\s)<r$ and for any $U\in B_{s_0}^K(I;r(\s))\cap C^{k+K'}_{*}(I,{\dot{H}}^{\s}(\T;\C^{2}))$ and any $\alpha, \beta \in\N$, with $\alpha\leq \s-s_0$
\begin{equation}\label{simbo}
|{\partial_t^k\partial_x^{\alpha}\partial_{\xi}^{\beta}a(U;t,x,\xi)}|\leq C\langle\xi\rangle^{m-\beta} 
\|U\|^{p-1}_{k+K',s_0}
\|{U}\|_{k+K',\s} \, .
\end{equation}

\noindent
$(iii)$ {\bf Symbols.} We denote by $\Sigma\Gamma^{m}_{K,K',p}[r,N]$
the space of functions 
$(U,t,x,\x)\to a(U;t,x,\x)$ such that there are homogeneous symbols $a_{q}\in \widetilde{\Gamma}_{q}^{m}$,
 $q=p,\ldots, N-1$, and a non-homogeneous symbol $a_{N}\in \Gamma^{m}_{K,K',N}[r]$
such that
$ a(U;t,x,\x)=\sum_{q=p}^{N-1}a_{q}(U,\ldots,U;x,\x)+a_{N}(U;t,x,\x) $.
We denote by $\Sigma\Gamma^{m}_{K,K',p}[r,N]\otimes \mathcal{M}_{2}(\C)$ the space $2\times 2$
matrices with entries  in  $\Sigma\Gamma^{m}_{K,K',p}[r,N]$.
\end{definition}
  
As a consequence of the 
\emph{momentum} condition \eqref{def:tr-in} a symbol
$ a_{1} $ in the class $\widetilde{\Gamma}_{1}^{m}$, for some 
$ m\in \mathbb{R} $,
can be written as
\begin{equation}\label{homosimbo2}
a_{1}(U;x,\x)= 
\sum_{ \substack{j \in  \Z \setminus \{0\}, \s =\pm } }
(a_{1})^{\s}_{j}(\x)
u_{j}^{\s}e^{\ii \s j x}
\end{equation}
for some coefficients 
$(a_{1})^{\s}_{j}(\x)\in \mathbb{C}$, see \cite{BFP}.

\begin{remark}\label{indipx}
A symbol $ a_1\in \widetilde{\Gamma}_{1}^{m}$
of the form \eqref{homosimbo2}, independent of $x$, is actually $ a_1\equiv0$.
\end{remark}

We also define classes of functions in analogy with our classes of symbols.

\begin{definition}{\bf (Functions)} \label{apeape} Fix $N, p \in \N$ with $p\leq N$,
 $K,K'\in \N$ with $K'\leq K$, $r>0$.
We denote by $\widetilde{\mathcal{F}}_{p}$, resp. $\mathcal{F}_{K,K',p}[r]$,  $\Sigma\mathcal{F}_{p}[r,N]$, 
the subspace of $\widetilde{\Gamma}^{0}_{p}$, resp. $\Gamma^{0}_{p}[r]$, 
resp. $\Sigma\Gamma^{0}_{p}[r,N]$, 
made of those symbols which are independent of $\x$.
We write $\widetilde{\mathcal{F}}^{\R}_{p}$, resp. $\mathcal{F}_{K,K',p}^{\R}[r]$, 
$\Sigma\mathcal{F}_{p}^{\R}[r,N]$,  to denote functions in $\widetilde{\mathcal{F}}_{p}$, 
resp. $\mathcal{F}_{K,K',p}[r]$,  $\Sigma\mathcal{F}_{p}[r,N]$, 
which are real valued.
\end{definition}

\noindent
{\bf Paradifferential quantization.}
Given $p\in \N$ we consider   functions
  $\chi_{p}\in C^{\infty}(\R^{p}\times \R;\R)$ and $\chi\in C^{\infty}(\R\times\R;\R)$, 
  even with respect to each of their arguments, satisfying, for $0<\delta\ll 1$,
\begin{align*}
&{\rm{supp}}\, \chi_{p} \subset\{(\xi',\xi)\in\R^{p}\times\R; |\xi'|\leq\delta \langle\xi\rangle\} \, ,\qquad \chi_p (\xi',\xi)\equiv 1\,\,\, \rm{ for } \,\,\, |\xi'|\leq \delta \langle\xi\rangle / 2 \, ,
\\
&\rm{supp}\, \chi \subset\{(\xi',\xi)\in\R\times\R; |\xi'|\leq\delta \langle\xi\rangle\} \, ,\qquad \quad
 \chi(\xi',\xi) \equiv 1\,\,\, \rm{ for } \,\,\, |\xi'|\leq \delta   \langle\xi\rangle / 2 \, . 
\end{align*}
For $p=0$ we set $\chi_0\equiv1$. 
We assume moreover that 
 $ |\partial_{\xi}^{\alpha}\partial_{\xi'}^{\beta}\chi_p(\xi',\xi)|\leq C_{\alpha,\beta}\langle\xi\rangle^{-\alpha-|\beta|} $, 
 $ \forall \alpha\in \N, \,\beta\in\N^{p} $, and 
$ |\partial_{\xi}^{\alpha}\partial_{\xi'}^{\beta}\chi(\xi',\xi)|\leq C_{\alpha,\beta}\langle\xi\rangle^{-\alpha-\beta} $,  $  \forall \alpha, \,\beta\in\N $. 

If $ a (x, \xi) $ is a smooth symbol 
we define its \emph{Weyl} quantization  as the operator
acting on a
$ 2 \pi $-periodic function
$u(x)$ (written as in \eqref{complex-uU})
 as
\be\label{Weil-Q}
Op^{W}(a)u=\frac{1}{\sqrt{2\pi}}\sum_{k\in \Z}
\Big(\sum_{j\in\Z}\hat{a}\big(k-j, \frac{k+j}{2}\big)\hat{u}(j) \Big)\frac{e^{\ii k x}}{\sqrt{2\pi}}
\ee
where $\hat{a}(k,\xi)$ is the $k^{th}-$Fourier coefficient of the $2\pi-$periodic function $x\mapsto a(x,\xi)$.

\begin{definition}{\bf (Bony-Weyl quantization)}\label{quantizationtotale}
If a is a symbol in $\widetilde{\Gamma}^{m}_{p}$, 
respectively in $\Gamma^{m}_{K,K',p}[r]$,
we set
\[
\begin{aligned}
& a_{\chi_{p}}(\mathcal{U};x,\x) := \sum_{\vec{n}\in \N^{p}}\chi_{p}\left(\vec{n},\x\right)a(\Pi_{\vec{n}}\mathcal{U};x,\x) \, , \\ 
& a_{\chi}(U;t,x,\x) :=\frac{1}{2\pi}\int_{\T}  \chi\left(\xi',\x\right)\hat{a}(U;t,\xi',\x)e^{\ii \xi' x}d \xi' \, ,
\end{aligned} 
\]
where in the last equality $  \hat a $ stands for the Fourier transform with respect to the $ x $ variable, and 
we define the \emph{Bony-Weyl} quantization of $ a $ as 
$$
\opbw(a(\mathcal{U};\cdot))=Op^{W}(a_{\chi_{p}}(\mathcal{U};\cdot)),\qquad
\opbw(a(U;t,\cdot))=Op^{W}(a_{\chi}(U;t,\cdot)) \, .
$$
If  $a$ is a symbol in  $\Sigma\Gamma^{m}_{K,K',p}[r,N]$, 
we define its \emph{Bony-Weyl} quantization 
$
\opbw(a(U;t,\cdot))=\sum_{q=p}^{N-1}\opbw(a_{q}(U,\ldots,U;\cdot))+\opbw(a_{N}(U;t,\cdot)) \, . 
$
\end{definition}

Paradifferential operators act on homogeneous spaces. 
If  $a$ is in  $\Sigma\Gamma^{m}_{K,K',p}[r,N]$, the corresponding
para-differential operator is bounded from $ {\dot H}^s $ to $ {\dot H}^{s-m} $, for all
$ s \in \R $, see Proposition 3.8 in \cite{BD}. 

Definition \ref{quantizationtotale} 
is  independent of the cut-off functions $\chi_{p}$, $\chi$,  
up to smoothing operators that we define below (see Definition $ 3.7 $ in \cite{BD}).
Roughly speaking, the class $\widetilde{\mathcal{R}}^{-\rho}_{p}$ contains smoothing operators
which gain $\rho$ derivatives and are homogeneous of degree $p$ in $U$, while the class 
$\mathcal{R}_{K,K',p}^{-\rho}$ contains non-homogeneous $\rho$-smoothing operators which
vanish at degree at least $p$ in $U$, and are $(K-K')$-times differentiable in $t$.

\noindent
Given  $(n_1,\ldots,n_{p+1})\in \N^{p+1}$ we denote by $\max_{2}(n_1 ,\ldots, n_{p+1})$ 
the second largest among the integers $ n_1,\ldots, n_{p+1}$.

 \begin{definition}{\bf (Classes of  smoothing operators)} \label{omosmoothing}
 Let $N\in \N^* $, $ K, K' \in \N $ with $K'\leq K\in\N$,  $\rho\geq0$ and $r>0$.

 (i) {\bf $p$-homogeneous smoothing operators.} We denote by $\widetilde{\mathcal{R}}^{-\rho}_{p}$
 the space of $(p+1)$-linear maps $R$ 
 from $(\dot{H}^{\infty}(\T;\C^{2}))^{p}\times \dot{H}^{\infty}(\T;\C)$ to 
$\dot{H}^{\infty}(\T;\C)$,  symmetric
 in $(U_{1},\ldots,U_{p})$, of the form
$ (U_{1},\ldots,U_{p+1})\to R(U_1,\ldots, U_p)U_{p+1}$
 that satisfy the following. There are $\mu\geq0$, $C>0$ such that 
$$
 \|\Pi_{n_0}R(\Pi_{\vec{n}}\mathcal{U})\Pi_{n_{p+1}}U_{p+1}\|_{L^{2}}\leq
 C\frac{\max_2( n_1,\ldots, n_{p+1})^{\rho+\mu}}{\max( n_1,\ldots, n_{p+1})^{\rho}}
 \prod_{j=1}^{p+1}\|\Pi_{n_{j}}U_{j}\|_{L^{2}}
$$
  for any 
 $\mathcal{U}=(U_1,\ldots,U_{p})\in (\dot{H}^{\infty}(\T;\C^{2}))^{p}$,  
 $U_{p+1}\in \dot{H}^{\infty}(\T;\C)$,
  $\vec{n}=(n_1,\ldots,n_p)\in (\N^{*})^{p}$, any $n_0,n_{p+1}\in \N^*$.
 Moreover, if 
 \begin{equation}\label{omoresti2}
 \Pi_{n_0}R(\Pi_{n_1}U_1,\ldots,\Pi_{n_{p}}U_{p})\Pi_{n_{p+1}}U_{p+1}\neq 0 \, ,
 \end{equation}
 then there is a choice of signs $\s_0,\ldots,\s_{p+1}\in\{\pm 1\}$ such that 
 $\sum_{j=0}^{p+1}\s_j n_{j}=0$. 
 In addition we require the translation invariance property
\be\label{def:R-trin}
R( \tau_\teta {\cal U}) [\tau_\teta U_{p+1}]  =  \tau_\theta \big( R( {\cal U})U_{p+1} \big) \, , \quad \forall \theta \in \R \, . 
\ee
(ii) {\bf Non-homogeneous smoothing operators.}
  We denote by $\mathcal{R}^{-\rho}_{K,K',N}[r]$ 
  the space of maps $(V,U)\mapsto R(V)U$ defined on $B^K_{s_0}(I;r)\times C^K_{*}(I,\dot{H}^{s_0}(\T,\C))$ 
  which are linear in the variable $U$ and such that the following holds true. 
  For any $s\geq s_0$ there are $C>0$ and $r(s)\in]0,r[$ such that, for any $V\in B^K_{s_0}(I;r)\cap C^K_{*}(I,\dot{H}^{s}(\T,\C^2))$, 
  any $  U \in C^K_{*}(I,\dot{H}^{s}(\T,\C))$, any 
  $0\leq k\leq K-K'$ and any $t\in I$, we have 
\begin{equation}
\begin{aligned} \label{piove}
\|{\partial_t^k\left(R(V)U\right)(t,\cdot)}\|_{\dot{H}^{s- \frac32 k+\rho}} 
& \leq \sum_{k'+k''=k}C\Big( \|{U}\|_{k'',s}\|{V}\|_{k'+K',s_0}^{N} \\
& \quad \qquad \  \quad +\|{U}\|_{k'',s_0}\|V\|_{k'+K',s_0}^{N-1}\|{V}\|_{k'+K',s}\Big) \, .
\end{aligned}
\end{equation}
(iii) {\bf Smoothing operators.} We denote by $\Sigma\mathcal{R}^{-\rho}_{K,K',p}[r,N]$
the space of maps $(V,t,U)\to R(V;t)U$ 
that may be written as 
$ R(V;t) U =\sum_{q=p}^{N-1}R_{q}(V,\ldots,V) U +R_{N}(V;t) U $
for some $R_{q} $ in $ \widetilde{\mathcal{R}}^{-\rho}_{q}$, $q=p,\ldots, N-1$ and $R_{N}$  in 
$\mathcal{R}^{-\rho}_{K,K',N}[r]$.

 We denote by  $\Sigma\mathcal{R}^{-\rho}_{K,K',p}[r,N]\otimes\mathcal{M}_2(\C)$
the space of $2\times2$ matrices with entries
  in the class $\Sigma\mathcal{R}^{-\rho}_{K,K',p}[r,\! N]$.
 \end{definition}

Below we introduce classes of operators without keeping track of the number of lost derivatives in a precise way 
(see Definition 3.9  in \cite{BD}).
The class $\widetilde{\mathcal{M}}^{m}_{p}$ denotes multilinear maps that lose $m$ derivatives
and are $p$-homogeneous in $U$, 
while the class $\mathcal{M}_{K,K',p}^{m}$ contains non-homogeneous maps which lose $m$ derivatives,
vanish at degree at least $p$ in $U$, and are $(K-K')$-times differentiable in $t$.

\begin{definition}{\bf (Classes of maps)} \label{smoothoperatormaps}
Let $p,N\in \N $, with $p\leq N$, $N\geq1$, $K,K'\in\N$ with $K'\leq K$
and $ m \geq 0 $. 

(i) {\bf $p$-homogeneous maps.} 
We denote by $\widetilde{\mathcal{M}}^{m}_{p}$
 the space of $(p+1)$-linear maps $M$ 
 from $(\dot{H}^{\infty}(\T;\C^{2}))^{p}\times \dot{H}^{\infty}(\T;\C)$ to 
 $\dot{H}^{\infty}(\T;\C)$ which are symmetric
 in $(U_{1},\ldots,U_{p})$, of the form
$ (U_{1},\ldots,U_{p+1})\to M(U_1,\ldots, U_p)U_{p+1} $
and that satisfy the following. There is $C>0$ such that 
 \[
 \|\Pi_{n_0}M(\Pi_{\vec{n}}\mathcal{U})\Pi_{n_{p+1}}U_{p+1}\|_{L^{2}}\leq
 C( n_0 +  n_1 +\cdots+ n_{p+1})^{m} \prod_{j=1}^{p+1}\|\Pi_{n_{j}}U_{j}\|_{L^{2}}
\] 
  for any 
 $\mathcal{U}=(U_1,\ldots,U_{p})\in (\dot{H}^{\infty}(\T;\C^{2}))^{p}$, any 
 $U_{p+1}\in \dot{H}^{\infty}(\T;\C)$,
 $\vec{n}=(n_1,\ldots,n_p) $ in $  (\N^*)^{p}$, any $ n_0,n_{p+1}\in \N^*$.
 Moreover the properties \eqref{omoresti2}-\eqref{def:R-trin} hold.

(ii) {\bf Non-homogeneous maps.}
  We denote by  $\mathcal{M}^{m}_{K,K',N}[r]$ 
  the space of maps $(V,u)\mapsto M(V) U $ defined on $B^K_{s_0}(I;r)\times C^K_{*}(I,\dot{H}^{s_0}(\T,\C))$ 
  which are linear in the variable $ U $ and such that the following holds true. 
  For any $s\geq s_0$ there are $C>0$ and 
  $r(s)\in]0,r[$ such that for any 
  $V\in B^K_{s_0}(I;r)\cap C^K_{*}(I,\dot{H}^{s}(\T,\C^2))$, 
  any $ U \in C^K_{*}(I,\dot{H}^{s}(\T,\C))$, any $0\leq k\leq K-K'$, $t\in I$, we have that
  $\|{\partial_t^k\left(M(V)U\right)(t,\cdot)}\|_{\dot{H}^{s-\frac{3}{2}k-m}}$ is bounded by the right hand side  of \eqref{piove}.

(iii) {\bf Maps.}
We denote by $\Sigma\mathcal{M}^{m}_{K,K',p}[r,N]$
the space of maps $(V,t,U)\to M(V;t)U$
that may be written as 
$ M(V;t)U=\sum_{q=p}^{N-1}M_{q}(V,\ldots,V)U+M_{N}(V;t)U $
for some $M_{q} $ in $ \widetilde{\mathcal{M}}^{m}_{q}$, $q=p,\ldots, N-1$ and $M_{N}$  in  
$\mathcal{M}^{m}_{K,K',N}[r]$.
Finally we set $\widetilde{\mathcal{M}}_{p}:=\cup_{m\geq0}\widetilde{\mathcal{M}}_{p}^{m}$,
$\mathcal{M}_{K,K',p}[r]:=\cup_{m\geq0}\mathcal{M}^{m}_{K,K',p}[r]$, 
$\Sigma\mathcal{M}_{K,K',p}[r,N]:=\cup_{m\geq0}\Sigma\mathcal{M}^{m}_{K,K',p}[r]$.

We denote by  $\Sigma\mathcal{M}_{K,K',p}^{m}[r,N]\otimes\mathcal{M}_2(\C)$
 the space of $2\times 2$ matrices whose entries are maps in
 $\Sigma\mathcal{M}^{m}_{K,K',p}[r,N]$.
We  set $\Sigma\mathcal{M}_{K,K',p}[r,N]\otimes\mathcal{M}_2(\C) :=\cup_{m\in \R}
\Sigma\mathcal{M}_{K,K',p}^{m}[r,N]\otimes\mathcal{M}_2(\C)$.
\end{definition}

Given an  operator $\mathtt{R}_{1}$ in 
$\widetilde{\mathcal{R}}^{-\rho}_{1}$ (or in $\widetilde{\mathcal{M}}^{m}_{1}$), 
and $ z^{\s_2} $, $ \s_2 = \pm $, 
 the \emph{momentum} condition \eqref{def:R-trin} implies 
 that
\begin{equation}\label{Homoop}
\mathtt{R}_{1}(U)[z^{\s_2}]= 
\sum_{j_1,j_{2}\in \mathbb{Z}\setminus\{0\}, \s_1 = \pm}
\!\!\!\!\!\!\!\! (\mathtt{R}_{1})^{\s_1, \s_{2}}_{j_1, j_{2}}
u_{j_{1}}^{\s_1} z_{j_{2}}^{\s_{2}}
e^{\ii(\s_1 j_1+\s_{2}j_{2})x}
\end{equation}
for some 
$ (\mathtt{R}_{1})^{\s_1,\s_{2}}_{j_1, j_{2}}
\in \mathbb{C} $, see \cite{BFP}.

\begin{proposition}{\bf (Compositions)} \label{composizioniTOTALI}
Let $m,m'\in \R$, $N, K,K' \in \N $ with $K'\leq K$, 
$ p_1,p_2,p_{3} \in \N$, 
$\rho\geq0$ and $r>0$.
Let $a\in \Sigma\Gamma^{m}_{K,K',p_1}[r,N]$, 
$R\in\Sigma\mathcal{R}^{-\rho}_{K,K',p_{2}}[r,N]$ and $M\in \Sigma\mathcal{M}^{m'}_{K,K',p_{3}}[r,N]$.
Then:

\noindent
$(i)$
$ R(U;t)\circ \opbw(a(U;t,x,\x)) $, $  \opbw(a(U;t,x,\x))\circ R(U;t) $
are   in $\Sigma\mathcal{R}^{-\rho+m}_{K,K',p_1+p_{2}}[r,N]$;

\noindent
$(ii)$
$ R(U;t)\circ M(U;t) $ and $ M(U;t)\circ R(U;t) $ 
are smoothing operators in $\Sigma\mathcal{R}^{-\rho+m'}_{K,K',p_2+p_{3}}[r,N]$;

\noindent
$(iii)$
If  $R_{1} \in \widetilde{\mathcal{R}}_{p_1}^{-\rho}$, $ p_1 \geq	 1 $,   
then $ R_{1} ( \underbrace{U, \ldots, U}_{p_1-1},  M(U;t)U)$ belongs to $\Sigma\mathcal{R}^{-\rho+m'}_{K,K',p_1 +p_3}[r,N]$.
\end{proposition}

\begin{proof}
See Propositions 3.16,  3.17  
in \cite{BD}.
The translation invariance properties for the composed operators and symbols 
in items (i)-(ii) follow as in \cite{BFP}.
\end{proof}

\noindent
{\bf Real-to-real operators.} 
Given a linear  operator  
$ R(U) [ \cdot ]$ acting on $ \C $ (it may be a 
smoothing operator in  $\Sigma\mathcal{R}^{-\rho}_{K,K',1}$ or
a map  in $\Sigma\mathcal{M}_{K,K',1}$) 
we associate the linear  operator  defined by  
$$
\ov{R}(U)[v] := \ov{R(U)[\ov{v}]} \, ,   \quad \forall\, v \in \C \, .
$$
We say that a matrix of operators acting on $ \C^2 $ is   \emph{real-to-real}, if it has the form 
\begin{equation}\label{vinello}
R(U) =
\left(\begin{matrix} R_{1}(U) & R_{2}(U) \\
\ov{R_{2}}(U) & \ov{R_{1}}(U)
\end{matrix}
\right) \, .
\end{equation}
 If  $R(U)$  is a  real-to-real matrix of operators 
then, given $V=\vect{v}{\ov{v}}$, the vector  $Z:=R(U)[V]$ has the form  $Z=\vect{z}{\bar{z}}$, i.e. the second component is the complex conjugated of the first one.

Given two linear operators $ A, B $ (either two operator-valued matrices acting on $ \C^2 $ as in \eqref{vinello}), we denote their commutator by  $ [A, B] = A B - B A $.

\noindent
$\bullet$  The notation $A\lesssim_{s} B$ means that 
 $A \leq C(s) B$
for some positive constant  $C(s) > 0 $.

\section{Paradifferential reduction to constant symbols 
 up to smoothing operators}
 
The first step in order to prove Theorem \ref{BNFtheorem}
is to write \eqref{WW1} in  paradifferential form, to symmetrize it, 
and reduce to paradifferential symbols which are constant in $ x $, see Proposition 
\ref{regolo}. 
These results are proved  in \cite{BD} (up to minor details). 
We denote the horizontal and vertical components of the velocity field at the free interface by
\begin{align*} 
& V =  V (\eta, \psi) :=  (\pa_x \Phi) (x, \eta(x)) = \psi_x - \eta_x B \, , 
\\
& B =  B(\eta, \psi) := (\pa_y \Phi) (x, \eta(x)) =  \frac{G(\eta) \psi + \eta_x \psi_x}{ 1 + \eta_x^2} \, , 
\end{align*}
and the ``good unknown'' of Alinhac 
\begin{equation}\label{omega0}
\omega := \psi-\opbw(B(\eta,\psi))\eta \, ,
\end{equation} 
as introduced in Alazard-Metivier \cite{AlM}. 
The function $B(\eta,\psi)$ belongs to
$\Sigma\mathcal{F}^{\mathbb{R}}_{K,0,1}[r,N]$, for any $N>0$ (see 
Proposition 7.4 in \cite{BD}).
Then, by the action of a paraproduct, 
if  $ \eta \in H^{s+ \frac14}_0 $ and $ \psi \in {\dot H}^{s-\frac14} $ then
the good unknown $ \omega $ is in $ {\dot H}^{s-\frac14} $. 

Define  the Fourier multiplier $\Lambda$ of order $-1/4$  as
\begin{equation}\label{compl1}
\Lambda:=\Lambda(D):=
\big(D\tanh(hD)\big)^{\frac{1}{4}}
\big(g+\kappa D^{2}\big)^{-\frac{1}{4}}    
\end{equation}
and consider the complex function
\begin{equation}\label{compVar}
u:=\frac{1}{\sqrt{2}}\Lambda\omega
+\frac{\ii}{\sqrt{2}} \Lambda^{-1}\eta\,, \quad
\eta=\frac{1}{\ii \sqrt{2}}\Lambda(u-\bar{u})\,,
\quad \omega=\frac{1}{\sqrt{2}}\Lambda^{-1}(u+\bar{u})
\end{equation}
where $ \Lambda^{-1} $ acts on functions modulo constants in itself.

Let $K \in \N $.
We first remark that, if 
$(\eta,\psi)$ solves the gravity-capillary system 
\eqref{WW1}, then the function 
 $u$ defined in \eqref{compVar} 
 satisfies, by Proposition $7.9$ in \cite{BD}, 
 for $ s \gg K $, as long as $ u $ stays in the unit ball of 
 $ {\dot H}^s (\T, \C) $, 
\be\label{lem:tempo}
\| \pa_t^k u \|_{\dot{H}^{s-\frac{3}{2}k}} \lesssim_{s,K} \| u \|_{\dot{H}^s} \, ,
\quad \forall \; 0\leq k\leq K\,. 
\ee
 As a consequence, if \eqref{hypoeta} holds then 
\begin{equation}\label{pallaU}
\sup_{t\in [-T,T]}\|\partial_{t}^{k}u\|_{\dot{H}^{s-\frac{3}{2}k}}\leq C_{s, K} 
 \bar \e\,,\quad
\forall\, 0\leq k\leq K\,.
\end{equation}

\begin{proposition}{\bf (Paradifferential complex form of the water waves equations)}
\label{Formulazione}
Let $ N, K  \in \N^* $, $ \rho > 0 $.
Assume that $(\eta,\psi)$ solves the gravity-capillary system 
\eqref{WW1}
and satisfy \eqref{hypoeta} for some $T>0$ and $s\gg K$.
Then
the function $U:=\vect{u}{\bar{u}}$, with $u$ defined  in \eqref{compVar},
solves
\begin{equation}\label{complEQ}
D_{t}U= \mk(D)E U+\opbw(A(U;t,x,\x))U+R(U;t)U\,, 
\ \  E :=\sm{1}{0}{0}{-1}\,,
\end{equation}
where $ D_{t} :=\frac{1}{\ii}\pa_t $ and 

\noindent
$\bullet$  $\mk(D)=\opbw(\mk(\x))$ where $\mk(\x)\in \widetilde{\Gamma}_0^{\frac{3}{2}}$ is the dispersion relation symbol defined in \eqref{dispersionLaw}.
\noindent
$\bullet$ 
the matrix of symbols $A(U;t,x,\x)
\in \Sigma\Gamma^{1}_{K,1,1}[r,N]\otimes\mathcal{M}_2(\mathbb{C})$
has the form
\begin{equation}\label{formA}
\begin{aligned}
A(U;t,x,\x)&=\big(\zeta(U;t,x)\mk(\x)
+\lambda_{\frac{1}{2}}(U;t,x,\x)\big)\sm{1}{0}{0}{-1}\\
&+\big(\zeta(U;t,x)\mk(\x)
+\lambda_{-\frac{1}{2}} (U;t,x,\x) \big)\sm{0}{-1}{1}{0}\\
&+\lambda_1(U;t,x,\x)\sm{1}{0}{0}{1}
+\lambda_0(U;t,x,\x)\sm{0}{1}{1}{0}
\end{aligned}
\end{equation}
where 

\noindent
 $\bullet$
the function $ \zeta(U;t,x) $ is in $ \Sigma {\mathcal F}_{K,0,1}^\R [r,N] $; 

\noindent
 $\bullet$
the symbols $\lambda_{j} (U;t,x,\x) $ are in  
$ \Sigma\Gamma^{j}_{K,1,1}[r,N]$, $j=1,0,1/2,-1/2$, 
and $\Im \lambda_j (U;t,x,\x) $ are in $ \Sigma\Gamma^{j-1}_{K,1,1}[r,N]$
for $j=1,1/2$;

\noindent
 $\bullet$  the matrix of smoothing operators $R(U;t)$ is
in
 $\Sigma\mathcal{R}^{-\rho}_{K,1,1}[r,N]\otimes\mathcal{M}_2(\mathbb{C})$;

\noindent
$\bullet$ the operators  $ \ii \opbw(A(U;t,x,\x))$
and $ \ii R(U;t)$
are real-to-real, according to \eqref{vinello}.
\end{proposition}

\begin{proof}
It is Corollary $7.7$ and Proposition $7.8$ in \cite{BD}.
The only difference is that 
$ U(x) $ is not 
even in $ x $. 
The property that the homogeneous components 
$A_{p}(U;t,x,\x)$, $R_{p}(U;t)$,  $p=1,\ldots,\!N$,  of the matrices 
$A(U;t,x,\x)$, $R(U;t)$ satisfy  \eqref{def:tr-in} and  \eqref{def:R-trin}
is  checked as in \cite{BFP}. 
\end{proof}

System \eqref{complEQ} has  the form 
\begin{equation}\label{WWFou}
D_{t}U=\mk(D)E U+M(U;t)U 
\end{equation}
where $ M(U;t) $ is a real-to-real map 
in $ \Sigma\mathcal{M}^{m_1}_{K,1,1}[r,N]\otimes\mathcal{M}_2(\C) $  for some 
$ m_1\geq 3/ 2 $
(using that  paradifferential operators and smoothing remainders 
are maps, see  (4.2.6) in \cite{BD}). 

As in \cite{BD},  since 
the dispersion law \eqref{dispersionLaw} is  \emph{super-linear},  
system \eqref{complEQ} 
can be transformed into a paradifferential diagonal system with a 
symbol constant in $ x $, up to smoothing terms. 

\begin{proposition}{\bf (Reduction to constant coefficients up to smoothing operators)}\label{regolo}
Fix $ \rho > 0  $ arbitrary.  
There exist $s_0>0$, $K':=K'(\rho)$ such that, 
for any $s\geq s_0$, for all $0<r \leq r_0(s)$ 
small enough,  for all $K\geq K'$
and  any solution $U\in B^K_s (I;r)$ of  \eqref{complEQ},  
there is a family of real-to-real, bounded, invertible linear maps $\mathfrak{F}^{\theta} (U)$, $\theta\in [0,1]$, 
such that the function 
$$ 
Z :=\vect{z}{\bar{z}}=(\mathfrak{F}^{\theta}(U))_{|\theta=1}[U] 
$$
solves the system 
\begin{equation}\label{finalsyst}
 \begin{aligned}
 D_{t} Z &= 
 \opbw\big((1+\underline{\zeta}(U;t))\mk(\x) E + H(U;t,\x)\big) Z
 +{R}(U;t)[Z]
 \end{aligned}
 \end{equation}
 where

\noindent
$\bullet$
 the function 
  $\underline{\zeta}(U;t)\in\Sigma\mathcal{F}^{\mathbb{R}}_{K,K',1}[r,N]$ 
  and the diagonal matrix of symbols
$H(U;t,\x)\in \Sigma\Gamma^{1}_{K,K',1}[r,N] \otimes {\mathcal M}_2 (\C) $ 
are {\rm independent} of $x$; 

\noindent
$\bullet$
the symbol 
 $\Im H(U;t,\x)$ belongs to $\Sigma\Gamma^{0}_{K,K',1}[r,N] \otimes {\mathcal M}_2 (\C) $;
  
\noindent
$\bullet$
${R}(U;t)$ is  matrix of smoothing operators in
  $ \Sigma\mathcal{R}^{-\rho}_{K,K',1}[r,N] \otimes {\mathcal M}_2 (\C) $ 

\noindent
$\bullet$
the operators  $ \ii \opbw(H(U;t,\x)) $
and $\ii R(U;t) $ are real-to-real, according to  \eqref{vinello};

\noindent
$ \bullet $ the map $\mathfrak{F}^{\theta}(U) $ 
satisfies, for all $  \,0\leq k\leq K-K' $, for any 
$V\in C^{K-K'}_{*\R}(I;\dot{H}^{s}(\T;\C^2))$, 
\begin{equation}\label{stimFINFRAK}
\| \pa_{t}^{k}\mathfrak{F}^{\theta}(U)[V]\|_{\dot{H}^{s-\frac{3}{2}k}}
+\| \pa_{t}^{k}(\mathfrak{F}^{\theta}(U))^{-1}[V]\|_{\dot{H}^{s-\frac{3}{2}k}}
\leq \|V\|_{k,s} \big( 1+ C_{s,r, K} \|U\|_{K,s_0} \big) 
\end{equation}
 uniformly in $\theta\in [0,1]$. Moreover the map $\mathfrak{F}^{\theta}(U) = U + \theta M_1 (U)[U] +
M_{\geq 2} (\theta; U)[U] $ where $ M_1 (U) $ is in  
$ \widetilde {\cal M}_{1} \otimes {\cal M}_2 ( \C ) $ 
and  $ M_{\geq 2} (\theta; U) \in {\cal M}_{K,K',2}[r] \otimes {\cal M}_2 ( \C ) $
with estimates uniform in $ \theta \in [0,1] $.  
\end{proposition}

\begin{proof}
This statement collects the results of Propositions $ 4.9$, $ 5.1$ and $5.5$ in \cite{BD}.
The remainder in (5.2.9) in \cite{BD} has the form \eqref{finalsyst} 
expressing $ U = (\mathfrak{F}^{\theta}(U))_{|\theta=1}^{-1} Z $ and using the 
estimates \eqref{stimFINFRAK}, which follow by Lemma 3.22 in \cite{BD}.  
Another difference is that  $ Z(x) $ is not even in $ x $. 
The $ x $-invariance properties \eqref{def:tr-in} for the symbols and  \eqref{def:R-trin}
for the smoothing operators are  checked as in \cite{BFP}. 
The last statement follows using Lemma A.2 in \cite{BFP}. 
\end{proof}

\section{Poincar\'e - Birkhoff normal form at quadratic degree}\label{sec:5}

From this section the analysis strongly differs from \cite{BD}. 

\begin{itemize}
\item 
{\bf Notation}: 
for simplicity 
 in the sequel we 
omit to write the dependence on the time $ t $ in the symbols,
smoothing remainders and maps, 
writing $a(U;x,\x)$, $ R(U) $, $ M(U) $ instead of 
$a(U;t,x,\x)$, $ R (U; t)$, $ M (U; t) $. 
\end{itemize} 

The aim of this section is to transform
system  \eqref{finalsyst} into its quadratic Poincar\'e-Birkhoff normal form, 
see system \eqref{finalsyst1012}. 
We first observe that the paradifferential vector field in  \eqref{finalsyst} of quadratic 
homogeneity is actually zero.
\begin{lemma}\label{lem:41} {\bf (Quadratic 
Poincar\'e-Birkhoff normal form 
up to smoothing vector fields)}
The system \eqref{finalsyst} with $N=2$ has the form
\begin{equation}\label{finalsyst800}
 \begin{aligned}
 \pa_{t} Z &= \ii \mk(D)E  Z +  \mathtt{R}_1(U)[Z]
+ \widetilde{\mathcal{X}}_{\geq3}(U,Z)
 \end{aligned}
 \end{equation}
 where 
 $ \mathtt{R}_1 (U) \in   \widetilde{\mathcal{R}}^{-\rho}_1\otimes\mathcal{M}_2(\mathbb{C})$ and
 \begin{equation}\label{mathcalX3}
 \widetilde{\mathcal{X}}_{\geq3}(U,Z) =  \ii \opbw\big(\mathcal{H}_{\geq2} (U;\x) \big)Z
 +  \mathtt{R}_{\geq2}(U)[Z] 
 \end{equation}
where  $ \mathcal{H}_{\geq2}(U;\x) \in \Gamma^{3/2}_{K,K',2}[r]\otimes\mathcal{M}_2(\mathbb{C})   $ 
is  a diagonal matrix of symbols independent of $x$,   
 such that 
 \begin{equation}\label{quasi-real}
 \Im \mathcal{H}_{\geq2}(U;\x)\in \Gamma^{0}_{K,K',2}[r] \otimes {\mathcal M}_2 (\C)\,,
  \end{equation}
and 
 ${\mathtt R}_{\geq2} (U) \in {\mathcal R}^{-\rho}_{K,K',2}[r] \otimes {\mathcal M}_2 (\C)$.
 The operators $\mathtt{R}_{1}(U)$ and 
  $ \widetilde{\mathcal{X}}_{\geq3}(U,Z)$ are real-to-real.
\end{lemma}

\begin{proof}
We expand  in homogeneity the function 
$\underline{\zeta}(U) =\zeta_1(U)+\zeta_{\geq2}(U)$, $\zeta_1\in \widetilde{\mathcal{F}}^{\mathbb{R}}_1$, 
the diagonal matrix of symbols 
$H(U;\x) =H_1(U;\x)+H_{\geq2}(U;\x)$, $H_{1}(U;\x) \in \widetilde{\Gamma}^{1}_{1}
\otimes\mathcal{M}_2(\mathbb{C})  $, 
 and the smoothing remainder
$ {R}(U)  = -\ii {\mathtt R}_1(U) -\ii {\mathtt R}_{\geq2}(U)$, 
${\mathtt R}_1(U) \in \widetilde{{\mathcal R}}^{-\rho}_{1} \otimes\mathcal{M}_2(\mathbb{C}) $.
Since  the function $\zeta_1 (U) $ and  $H_1 (U; \xi) $ 
admit an expansion as \eqref{homosimbo2} and 
are independent of $ x $ (see Proposition \ref{regolo}),  
 Remark \ref{indipx} implies that 
$\zeta_1(U) = 0 $, $  H_1(U;\x) = 0 $.  This proves  \eqref{finalsyst800}-\eqref{quasi-real}.
\end{proof}

System \eqref{finalsyst800} 
 is yet  in Poincar\'e-Birkhoff normal form 
at degree 2  up to smoothing remainders
and   the cubic term 
$ \widetilde{{\mathcal X}}_{\geq 3} $  in \eqref{mathcalX3} admits
an   energy estimate as \eqref{theoBirR}, since 
${\mathcal{H}}_{\geq2}(U;  \xi) $
is independent of $x$ and purely imaginary up to symbols of order $ 0 $,   
see \eqref{quasi-real}.

The goal  is now to transform  the quadratic smoothing term 
$ \mathtt{R}_1(U)[Z] $  in 
\eqref{finalsyst800} to Poincar\'e-Birkhoff normal form at degree $ 2 $, see
Definition \ref{Resonant}.
The remainder $\mathtt{R}_1(U)$ in \eqref{finalsyst800}
is  \emph{real-to-real} (i.e. has the form \eqref{vinello}), 
satisfies the momentum condition \eqref{def:R-trin}, thus it has the form 
\eqref{Homoop}, 
and so  we write it as  
 \begin{equation}\label{smooth-terms2} 
\mathtt{R}_1(U) 
 =\left(
\begin{matrix}
(\mathtt{R}_1(U))_{+}^{+} & (\mathtt{R}_1(U))_{+}^{-}\\
(\mathtt{R}_1(U))_{-}^{+} & (\mathtt{R}_1(U))_{-}^{-}
\end{matrix}
\right)\,,  \ 
(\mathtt{R}_1(U))_{\s}^{\s'}
\in \widetilde{\mathcal{R}}^{-\rho}_1 \, ,  \  
 (\mathtt{R}_{1}(U))_{\s}^{\s'}=\ov{ (\mathtt{R}_{1}(U))_{-\s}^{-\s'}} \, , 
 \end{equation}
 for $\s,\s'=\pm$.
 For any $\s,\s'=\pm$ we expand
\begin{equation}\label{BNF2}
(\mathtt{R}_1(U))_{\s}^{\s'} =\sum_{\ep=\pm}(\mathtt{R}_{1,\ep}(U))_{\s}^{\s'} \,, 
\end{equation}
where, for $ \ep=\pm$,  and $(\mathtt{R}_{1,\ep}(U))_{\s}^{\s'}\in 
\widetilde{\mathcal{R}}^{-\rho}_{1} $
is the 
homogeneous smoothing operator
\begin{align}\label{espansio1}
(\mathtt{R}_{1,\ep}(U))_{\s}^{\s'} z^{\s'} & = 
 \frac{1}{\sqrt{2\pi}}\sum_{j\in \Z\setminus\{0\}} 
\Big( \sum_{k\in \Z\setminus\{0\}}  (\mathtt{R}_{1,\ep}(U))_{\s,j}^{\s',k} z_{k}^{\s'}\Big) e^{\ii \s jx} 
\end{align}
with entries 
 \begin{align}\label{BNF3}
&  (\mathtt{R}_{1,\ep}(U))_{\s,j}^{\s',k} :=
\frac{1}{\sqrt{2\pi}}
\sum_{\substack{n\in \Z\setminus\{0\} \\ \ep n+\s'k=\s j}}
 (\mathtt{r}_{1,\ep})^{\s,\s'}_{n,k}u_{n}^{\ep} \, , 
 \quad j,k\in \Z\setminus\{0\} \, , 
 \end{align}
 for suitable scalar coefficients 
$ (\mathtt{r}_{1,\ep})^{\s,\s'}_{n,k} \in \C $. The restriction
$\ep n+\s'k=\s j $ is due to the momentum condition.

\begin{definition}
{\bf (Poincar\'e-Birkhoff Resonant smoothing operator)}\label{Resonant}
Given a real-to-real,  smoothing operator
$\mathtt{R}_{1} (U) \in \widetilde{\mathcal{R}}^{-\rho}_{1}
\otimes\mathcal{M}_2(\mathbb{C})$
as in \eqref{smooth-terms2}-\eqref{BNF3}, we define the Poincar\'e-Birkhoff resonant,  
real-to-real, smoothing operator
$\mathtt{R}_{1}^{res} (U) \in \widetilde{\mathcal{R}}^{-\rho}_{1}
\otimes\mathcal{M}_2(\mathbb{C})$
with  matrix entries 
 $(\mathtt{R}_{1,\ep}^{res} (U))_{\s,j}^{\s',k}$  defined as in \eqref{BNF3} such that, 
   for any  $\ep,\s,\s'=\pm $, $ j,k\in \Z\setminus\{0\} $, 
\begin{align}\label{RES1}
&  (\mathtt{R}_{1,\ep}^{res} (U))_{\s,j}^{\s',k} =
\frac{1}{\sqrt{2\pi}}
\sum_{\substack{n\in \Z\setminus\{0\}\, , \ep n+\s'k=\s j \\
\s \mk(j) - \s' \mk(k) - \ep \mk(n)  =0}}
 (\mathtt{r}_{1,\ep})^{\s,\s'}_{n,k}u_{n}^{\ep} \, .   
 \end{align} 
 \end{definition}


In the next Proposition we conjugate \eqref{finalsyst800} into  its complete quadratic 
Poincar\'e-Birkhoff normal form.

    \begin{proposition}{\bf (Quadratic Poincar\'e-Birkhoff normal form)}\label{cor:BNF10}
 There exists $ \rho_0 > 0 $ such that, for all $ \rho \geq \rho_0 $, $ K \geq K' $
 with $K' := K'  (\rho) $ given by Proposition \ref{regolo}, 
there exists $s_0>0$ such that, 
for any $s\geq s_0$, for all $0<r \leq r_0(s)$ small enough,  
and  any solution $U\in B^K_s(I;r)$ of the water waves system 
\eqref{complEQ}, there is a family of 
 real-to-real, bounded, invertible linear maps
$ \mathfrak{C}^{\theta}(U) $, $ \theta \in [0,1] $, 
such that, if $ Z $ solves \eqref{finalsyst800}, then the function
$$
Y :=  \vect{y}{\bar{y}} = (\mathfrak{C}^{\theta}(U)[Z])_{|\theta=1}
$$ 
solves
\begin{equation}\label{finalsyst1012}
 \pa_{t}Y = \ii \mk(D)E Y+\mathtt{R}_1^{res}(Y)[Y]
 +  \mathcal{X}_{\geq3}(U,Y)
 \end{equation}
where:

\noindent
$ \bullet $ $ E$ is the matrix in \eqref{complEQ} and $\mk(D)$ has symbol  \eqref{dispersionLaw};

\noindent
$ \bullet $
${\mathtt{R}}_1^{res} (Y) \in 
 \widetilde{\mathcal{R}}^{- (\rho - \rho_0)}_1
 \otimes\mathcal{M}_2(\C) $ 
  is the real-to-real  Poincar\'e-Birkhoff resonant smoothing operator introduced in  
 Definition \ref{Resonant};

\noindent
$ \bullet $  $ \mathcal{X}_{\geq3}(U,Y) $ has the form 
 \begin{equation}\label{Stimaenergy100}
 \mathcal{X}_{\geq3}(U,Y)
 =\vect{\mathcal{X}^{+}_{\geq3}(U,Y)}{\ov{\mathcal{X}_{\geq3}^{+}(U,Y)}}
 :=
 \ii\opbw({\mathcal{H}}_{\geq2}(U;  \xi) )[Y]
 +\mathfrak{R}_{\geq 2}(U)[Y]
 \end{equation}
where 
 ${\mathcal{H}}_{\geq2}(U;  \xi )$ is defined in \eqref{mathcalX3}
 and satisfies \eqref{quasi-real},
 while $\mathfrak{R}_{\geq 2}(U)$ is a matrix 
 of real-to-real smoothing operators in 
 $\mathcal{R}^{-(\rho - \rho_0)}_{K,K',2}[r]\otimes\mathcal{M}_2(\C)$;

\noindent 
$ \bullet $  the map $\mathfrak{C}^{\theta}(U)$ 
satisfies, for any $0\leq k\leq K-K'$,  $V\in C^{K-K'}_{*\R}(I;\dot{H}^{s}(\T;\C^2)) $, 
\begin{equation}\label{stimafinaleFINFRAK}
\begin{aligned}
& \| \pa_{t}^{k}\mathfrak{C}^{\theta}(U)[V]\|_{\dot{H}^{s-\frac{3}{2}k}}+\| \pa_{t}^{k}(\mathfrak{C}^{\theta}(U))^{-1}[V]\|_{\dot{H}^{s-\frac{3}{2}k}} \\ 
& \leq \|V\|_{k,s}(1+C_{s,r,K} \|U\|_{K,s_0}) 
+ C_{s,r,K} \|V\|_{k,s_0}\|U\|_{K,s} \,,
\end{aligned}
\end{equation}
uniformly in $\theta\in [0,1]$.
Moreover the map $\mathfrak{C}^{\theta}(U) = U + \theta M_1 (U)[U] +
M_{\geq 2} (\theta; U)[U] $ where $ M_1 (U) $ is in  
$ \widetilde {\cal M}_{1} \otimes {\cal M}_2 ( \C ) $ 
and  $ M_{\geq 2} (\theta; U) \in {\cal M}_{K,K',2}[r] \otimes {\cal M}_2 ( \C ) $
with estimates uniform in $ \theta \in [0,1] $. 
\end{proposition}

In order to prove Proposition \ref{cor:BNF10} we first
 provide lower bounds on the 
``small divisors'' which appear in the Poincar\'e-Birkhoff reduction procedure.

 \subsection{Three waves interactions}

We  analyze the possible three waves interactions among the
 linear frequencies \eqref{dispersionLaw}.
  We first notice that they admit an expansion as 
 \be\label{omega-n}
\mk(n)= 
\sqrt{ |n| \tanh(h|n|)(g +\kappa n^{2})} = 
\sqrt{\kappa } |n|^{\frac32} + {\mathtt r} (n ) \, , \quad 
|{\mathtt r} (n )| \leq  C |n|^{-\frac{1}{2}}
\ee
for some constant $ C := C(g,\kappa, h) >  0 $.

\begin{lemma}{\bf ($ 3 $-waves interactions)}
\label{stimedalbasso}
There exist $ \mathtt{c} $, $ \mathtt{C} > 0 $ such that   
for any $n_1,n_2,n_3\in \mathbb{Z}\setminus\{0\}$, 
$\s,\s'=\pm$,
such that 
\begin{equation}\label{moment}
 n_1+\s n_2+\s' n_3=0 \,, 
 \end{equation}
and $ \max (|n_1|, |n_2|, |n_3|) \geq  \mathtt{C}  $, we have 
\begin{equation}\label{stima1}
|\mk(n_1)+\s \mk(n_2)+\s'\mk(n_3)|\geq \mathtt{c} \, . 
\end{equation}
If $ \max (|n_1|, |n_2|, |n_3|) < \mathtt{C}  $, then, either 
the phase 
$ \mk(n_1)+\s \mk(n_2)+\s'\mk(n_3) $ is zero, or \eqref{stima1} holds.
 \end{lemma}
 
 \begin{proof} 
If $\s=\s'=+$ then the bound \eqref{stima1} is trivial for all $ n_1, n_2, n_3 \in \Z \setminus 
\{ 0\} $. 
Assume $\s= - $ and $\s'=-$ (the cases $(\s,\s')=(+,-)$ and $(\s,\s')=(-,+)$ are the same, up to reordering the indexes). 
Then, by  \eqref{moment}, we have  $ n_1 =  n_2 + n_3  $ 
and we may suppose that  $ | n_1| \geq |n_2|, | n_3 | $, otherwise
the bound \eqref{stima1} is trivial. Without loss of generality we assume
$n_1>0$, 
thus, also  
$ n_2 $ and $ n_3 $ are positive. In conclusion
we assume that 
$ n_1 \geq n_2 \geq n_3 \geq 1  $. By \eqref{omega-n}, 
 \begin{align}
 |\mk(n_1) - \mk(n_2)  -\mk(n_3) | 
 & =  | \mk(n_2+ n_3) - \mk(n_2)  - \mk(n_3) |\nonumber \\
 &  \geq \sqrt{\kappa} 
 \big(  (n_2 + n_3)^{\frac32} -  n_2^{\frac32} -  n_3^{\frac32}  \big)  - 
  \frac{3C}{\sqrt{n_3}} \, . \label{lobb}
 \end{align}
 Now 
 \begin{align}
(n_2+n_3)^{\frac32} -  n_2^{\frac32} -  n_3^{\frac32}  
 &
 = 
 \frac{ (n_2+n_3)^{3} -  (n_2^{\frac32} +  n_3^{\frac32})^2 }{  
  (n_2+n_3)^{\frac32} + n_2^{\frac32} +  n_3^{\frac32}  }  
 =
   \frac{  3 (n_2^2 n_3 +  n_2 n_3^2)
   - 2 n_2^{\frac32} n_3^{\frac32}    }{  
  (n_2+n_3)^{\frac32} + n_2^{\frac32} +  n_3^{\frac32}  }  \nonumber \\
  & 
   =
   \frac{  9 (n_2^2 n_3 +  n_2 n_3^2)^2
   - 4 n_2^3 n_3^3    }{  
  (n_2+n_3)^{\frac32} + n_2^{\frac32} +  n_3^{\frac32}  } 
     \frac{  1   }{  
 3 (n_2^2 n_3 +  n_2 n_3^2) + 2 n_2^{\frac32} n_3^{\frac32}  }   \nonumber \\
 & = 
  \frac{  9 (n_2^4 n_3^2 +  n_2^2 n_3^4 ) 
   + 14 n_2^3 n_3^3    }{  
\big(  (n_2+n_3)^{\frac32} + n_2^{\frac32} +  n_3^{\frac32} \big) \big(  3 (n_2^2 n_3 +  n_2 n_3^2) + 2 n_2^{\frac32} n_3^{\frac32} \big) }  \nonumber \\
& 
\geq \frac{9}{16(1+\sqrt 2)}\sqrt{n_2} \geq \frac{\sqrt{n_2}}{5}  \label{lb3}
 \end{align} 
using that $ n_2 \geq n_3 \geq	 1 $. By \eqref{lobb} and \eqref{lb3} we deduce that the phase
\be\label{lbfin}
|\mk(n_1) - \mk(n_2)  - \mk(n_3) | \geq 
\sqrt{n_2} \big( \frac{\sqrt{\kappa}}{5}  -   \frac{3C}{\sqrt{n_2 \, n_3}} \big) 
\geq 
\sqrt{n_2}  \frac{\sqrt{\kappa}}{10} 
\ee
if $ n_2 n_3 \geq ( 30 C)^2 / \kappa  $, in particular, since $ n_3 \geq 1 $, if
$$
n_2 \geq  C_1 := (30 C)^2 / \kappa \, .
$$  
Recall that $ n_1 = n_2 + n_3 \leq 2 n_2 $. Therefore $ n_2 \geq n_1 / 2 $ and
we conclude that
$$ 
n_1 = \max( n_1, n_2, n_3 ) \geq 2 C_1   \ 
 \Longrightarrow \  n_2 \geq C_1  
 \ 
 \Longrightarrow \ 
| \mk(n_1) - \mk(n_2)  - \mk(n_3) | \geq 
\sqrt{n_2}  \frac{\sqrt{\kappa}}{10} 
 \, .
$$
For the finitely many integers $ n_1, n_2, n_3 $
 satisfying 
$ \max (|n_1|, |n_2|, | n_3|) \leq {\mathtt C} :=2C_1  $ such that the phase 
$ \mk(n_1) - \mk(n_2) - \mk(n_3) \neq 0 $, the lower bound 
\eqref{stima1} is trivial. 
\end{proof}

\begin{remark}\label{rem:no-reso}
The constant $ C (g, \kappa, h) $ in \eqref{omega-n} is bounded by 
$ c ( \sqrt{\kappa} \,  h^{- 2} + g \kappa^{-1/2} ) $, for some constant
$ c > 0 $ independent  of $ \kappa, g, h $.
Then, there are $ h_0, \kappa_0  $ such that, 
if $ h \geq h_0 $,  $ \kappa > \kappa_0 g $, then \eqref{lbfin} holds, for all
$ n_1, n_2, n_3 \in \Z \setminus \{0\}  $. 
As a consequence there are no $ 3 $-waves interactions, i.e. 
\eqref{stima1} holds for all 
$ n_1, n_2, n_3 \in \Z \setminus \{0\} $. 
\end{remark}

Notice that, for some  values of the parameters $ (g, \kappa, h) $, 
there could be $ 3 $-waves interactions.

\subsection{Poincar\'e-Birkhoff normal form of the 
smoothing quadratic terms}
\label{primostepBNF}

In order to prove Proposition \ref{cor:BNF10},
  we conjugate \eqref{finalsyst800} 
  with  the flow 
 \be\label{BNFstep1}
\partial_{\theta} \mathfrak{C}^{\theta}(U)  = 
\mathtt{G}_1(U) \mathfrak{C}^{\theta}(U) \, , 
\quad \mathfrak{C}^{0}(U) = {\rm Id} \, ,  
\ee
with an operator 
$ \mathtt{G}_1(U) $ in $  \widetilde{\mathcal{R}}^{-\rho}_1 
\otimes\mathcal{M}_2(\C)$, 
of the same form of 
$\mathtt{R}_1(U)$ in 
\eqref{smooth-terms2}-\eqref{BNF3}, 
to be determined. 
We introduce the new variable
$ Y:=\vect{y}{\ov{y}} = \big(\mathfrak{C}^{\theta}(U)[Z]\big)_{|_{\theta=1}} $.

\begin{lemma} \label{lem:BNF1}
If $\mathtt{G}_1(U)\in 
\widetilde{\mathcal{R}}^{-\rho}_1
\otimes\mathcal{M}_{2}(\C)$ solves 
the homological equation
\begin{equation}\label{omoBNF}
\mathtt{G}_1(\ii \mk(D)EU)+ \big[  \mathtt{G}_1( U), 
\ii \mk(D)E \big]+ \mathtt{R}_1(U)= \mathtt{R}_1^{res}(U) \, ,
\end{equation}
where $\mathtt{R}_1^{res}(U)$ is the Poincar\'e-Birkhoff  resonant operator 
in Definition \ref{Resonant},
then  
\begin{equation}\label{BNF12}
 \pa_{t}Y= \ii \mk(D)E  Y+
 \mathtt{R}_1^{res}(U)[Y]
 +
 \ii\opbw\big(
  \mathcal{H}_{\geq2}(U;\x)\big)Y
  + {\mathtt{R}}_{\geq2}(U)[Y]
 \end{equation}
where  
$ \mathcal{H}_{\geq2}(U;\x) $ is the same diagonal matrix of 
symbols  in \eqref{mathcalX3} and
$ {\mathtt{R}}_{\geq 2}(U)$ is a real-to-real  smoothing operator in  
$\mathcal{R}^{-\rho + m_1}_{K,K',2}[r]\otimes\mathcal{M}_{2}(\C) $
with  $m_1\geq3/2$ (fixed below \eqref{WWFou}). 

The flow  map $\mathfrak{C}^{\theta}(U)$ in \eqref{BNFstep1}
satisfies  \eqref{stimafinaleFINFRAK} 
and $\mathfrak{C}^{\theta}(U) = U + \theta M_1 (U)[U] +
M_{\geq 2} (\theta; U)[U] $ where $ M_1 (U) $ is in  
$ \widetilde {\cal M}_{1} \otimes {\cal M}_2 ( \C ) $ 
and  $ M_{\geq 2} (\theta; U) \in {\cal M}_{K,K',2}[r] \otimes {\cal M}_2 ( \C ) $
with estimates uniform in $ \theta \in [0,1] $. 
 \end{lemma}
 
 \begin{proof}
 Since $\mathtt{G}_1(U)$ is a smoothing operator 
 then  the flow in \eqref{BNFstep1}
 is well-posed in Sobolev spaces 
 and satisfies the estimates \eqref{stimafinaleFINFRAK}, 
 as well as the last statement,  by  e.g. Lemma $A.3$ in \cite{BFP}.
To conjugate \eqref{finalsyst800} 
we apply the usual Lie expansion up to the first order 
(see for instance 
Lemma $A.1$ in \cite{BFP}). 
Denoting $ {\rm Ad}_{\mathtt{G}_1} := [ \mathtt{G}_1 , \ ] $, we have 
   \begin{align}
   \mathfrak{C}^{1}(U)\mk(D)E (\mathfrak{C}^{1}(U))^{-1} & =
 \mk(D)E+ \big[ \mathtt{G}_1(U),\mk(D)E \big]   \nonumber\\
 & +\int_{0}^{1} (1- \theta)
\mathfrak{C}^{\theta}(U) {\rm Ad}_{\mathtt{G}_1(U)}^2
[\mk(D)E] (\mathfrak{C}^{\theta}(U))^{-1} d \theta \label{core4} \, .
\end{align}
Using that $\mathtt{G}_1(U)$ belongs to 
 $\widetilde{\mathcal{R}}_{1}^{-\rho} \otimes\mathcal{M}_2(\C) $, Proposition \ref{composizioniTOTALI} and \eqref{stimafinaleFINFRAK}, 
 the integral term in  \eqref{core4}
 is a smoothing operator  in $ \mathcal{R}^{-\rho+
  \frac32}_{K,K',2}[r]\otimes\mathcal{M}_2(\C)$.
Similarly, we obtain
$$
\mathfrak{C}^{1}(U) \opbw(\mathcal{H}_{\geq2}(U;\xi)) 
 (\mathfrak{C}^{1}(U))^{-1}
=\opbw(\mathcal{H}_{\geq2}(U;\xi))
$$
 up to a matrix of smoothing operators in 
 $\mathcal{R}^{-\rho+\frac{3}{2}}_{K,K',2}[r]
 \otimes\mathcal{M}_2(\C)$.
Finally 
$$
  \mathfrak{C}^{1}(U)\big( \mathtt{R}_1(U)+
  \mathtt{R}_{\geq2}(U)\big)(\mathfrak{C}^{1}(U))^{-1}=
  \mathtt{R}_1(U) 
$$
plus a smoothing operator in 
$ \mathcal{R}^{-\rho}_{K,K',2}[r]
\otimes\mathcal{M}_2(\C)$.
 
Next we consider the contribution coming from the conjugation of $ \pa_t $. 
Applying again a Lie expansion formula 
(see Lemma $A.1$ in \cite{BFP}) 
 we get
 \begin{align}
&  \pa_t \mathfrak{C}^{1} (U) (\mathfrak{C}^{1} (U))^{-1} =  
  \pa_t\mathtt{G}_1(U) \, + 
  \nonumber
\\
 & 
 \frac{1}{2} \big[ \mathtt{G}_1(U), \pa_t \mathtt{G}_1(U) \big] 
  +\frac{1}{2}\int_0^{1} (1-\theta)^{2} \mathfrak{C}^{\theta}(U)
{\rm Ad}_{\mathtt{G}_1(U)}^2\left[ \partial_t \mathtt{G}_1(U) \right]
 (\mathfrak{C}^{\theta} (U))^{-1} d \theta \, . \label{core5}
\end{align}
Recalling \eqref{WWFou} we have 
 \begin{equation}\label{core8}
 \pa_t  \mathtt{G}_1(U) =  \mathtt{G}_1(\ii \mk(D)EU+\ii M(U)[U]) =
   \mathtt{G}_1(\ii \mk(D)EU)
 \end{equation}
 up to a term in $ \mathcal{R}^{-\rho+m_1}_{K,K',2}[r]\otimes\mathcal{M}_{2}(\C)$, where we used Proposition \ref{composizioniTOTALI}.
 By  \eqref{core8}, the fact that 
$ \mathtt{G}_1(\ii \mk(D)EU) $ is in 
$ \widetilde{\mathcal{R}}_{1}^{-\rho + (3/2) } \otimes\mathcal{M}_2(\C)   $ and \eqref{stimafinaleFINFRAK},   
we deduce that the term in \eqref{core5} belongs to 
$ \Sigma\mathcal{R}^{-\rho+m_1}_{K,K',2}[r,N]
\otimes\mathcal{M}_{2}(\C)$.
Collecting all the previous expansions,  and  
using  that  $ \mathtt{G}_1 (U) $ solves \eqref{omoBNF},  
we deduce \eqref{BNF12}. 
\end{proof}

We now solve the homological equation \eqref{omoBNF}.

\begin{lemma}{\bf (Homological equation)} \label{LemmaOMO}
Let  $ \mathtt{G}_{1}(U) $ be an operator of the 
form \eqref{smooth-terms2}-\eqref{BNF3} 
with coefficients 
\begin{equation}\label{omoBNF5}
  (\mathtt{g}_{1,\ep})^{\s,\s'}_{n,k} :=
  \frac{(\mathtt{r}_{1,\ep})^{\s,\s'}_{n,k}}{\ii 
  \big( \s \mk(j)-\s'  \mk(k)- \ep \mk(n) \big)} \, ,  
\end{equation}
for any  $\s, \s', \ep  = \pm $, $ j, n, k \in \Z \setminus \{0\} $,  
satisfying
\be\label{def:divBNF}
\s j -\s' k-\ep n=0\,, \qquad 
\s \mk(j) - \s' \mk(k) - \ep \mk(n)  \neq 0 \, , 
\ee
and $ (\mathtt{g}_{1,\ep})^{\s,\s'}_{n,k}:=0$ otherwise. Then 
$\mathtt{G}_1(U)$ is in 
$ \widetilde{\mathcal{R}}_1^{-\rho}\otimes\mathcal{M}_2(\C)  $ and
solves the homological equation \eqref{omoBNF}. 
\end{lemma}

\begin{proof}
The coefficients in \eqref{omoBNF5} are well defined 
by \eqref{def:divBNF} and, 
 by Lemma \ref{stimedalbasso}, they 
 satisfy the uniform lower bound 
$ | \s \mk(j) - \s' \mk(k) - \ep \mk(n)| \geq {\mathtt c } $. 
Then the operator $\mathtt{G}_1(U)$
is in $\widetilde{\mathcal{R}}_1^{-\rho}\otimes\mathcal{M}_2(\C)$, see 
e.g.  Lemma $6.5$ of \cite{BFP}.

Next, recalling \eqref{smooth-terms2},  the homological 
equation \eqref{omoBNF} amounts to the equations
 \begin{equation*}
(\mathtt{G}_1(\ii \mk(D)EU))_{\s}^{\s'}
+  (\mathtt{G}_1( U))_{\s}^{\s'}  \s' \ii \mk(D) -
 \s\ii \mk(D)  (\mathtt{G}_1( U))_{\s}^{\s'} +
(\mathtt{R}_1(U))_{\s}^{\s'}=
\big(\mathtt{R}_{1}^{res}(U) \big)_{\s}^{\s'}
\end{equation*}
for $  \s,\s'=\pm $, and, setting 
$ {\mathtt F}_1 (U)  := \mathtt{G}_{1}(\ii \mk(D)EU) $ to the equations, for any 
$j,k\in \Z\setminus\{0\}$, $\ep=\pm$, 
\begin{equation}\label{omoBNF3}
\begin{aligned}
&  ({\mathtt F}_{1,\ep} (U))_{\s,j}^{\s',k}+  (\mathtt{G}_{1,\ep}( U))_{\s,j}^{\s',k} \big(-
\s\ii  \mk(j)+\s' \ii \mk(k)
\big)
+
(\mathtt{R}_{1,\ep}(U))_{\s,j}^{\s', k}\\ 
& =
\big(\mathtt{R}_{1,\ep}^{res}(U) \big)_{\s,j}^{\s', k}
 \, .
 \end{aligned}
\end{equation}
Expanding 
 $(\mathtt{G}_{1}(U))_{\s}^{\s'}$ as 
 in \eqref{BNF2}-\eqref{BNF3} with entries 
$$
(\mathtt{G}_{1,\ep}(U))_{\s,j}^{\s',k}
=\frac{1}{\sqrt{2\pi}}
\sum_{\substack{n\in \Z\setminus\{0\} \\ \ep n+\s'k=\s j}}
 (\mathtt{g}_{1,\ep})^{\s,\s'}_{n,k}u_{n}^{\ep} \, , \quad j,k\in \Z\setminus\{0\}\, , 
$$
we have  that $ {\mathtt F}_1 (U)  := \mathtt{G}_{1}(\ii \mk(D)EU) $ satisfies 
 \[
 (\mathtt{F}_{1,\ep}(U))_{\s,j}^{\s',k}=\frac{1}{\sqrt{2\pi}}
 \sum_{n\in \Z\setminus\{0\}, \ep n+\s'k=\s j}
 (\mathtt{g}_{1,\ep})^{\s,\s'}_{n,k} (\ii \mk(n) \ep ) u_{n}^{\ep}  \, .
 \]
 Hence the left hand side in \eqref{omoBNF3} has coefficients
 \begin{equation*} 
  -(\mathtt{g}_{1,\ep})^{\s,\s'}_{n,k} 
  \ii \big(\s \mk(j)  
-\s' \mk(k)- \ep \mk(n)\big)+
   (\mathtt{r}_{1,\ep})^{\s,\s'}_{n,k}  
 \end{equation*}
  for $j,k,n\in \Z\setminus\{0\}$ and  $\s,\s',\ep=\pm$ with 
 $ \ep n+\s'k=\s j$.
 Recalling Definition \ref{Resonant} 
 we deduce that $\mathtt{G}_1(U)$ with coefficients in \eqref{omoBNF5}
solves  the homological equation \eqref{omoBNF}. 
\end{proof}

\begin{proof}[{\bf Proof of Proposition \ref{cor:BNF10}}]
We apply Lemmata \ref{lem:BNF1} and \ref{LemmaOMO}. 
The change of variables that transforms \eqref{finalsyst800} into \eqref{BNF12} is 
$ Y = \mathfrak{C}^{\theta}(U) Z $ where $ \mathfrak{C}^{\theta}(U)$
is the flow map  in \eqref{BNFstep1} that satisfies \eqref{stimafinaleFINFRAK}
and the last statement in Lemma \ref{lem:BNF1}. 
Moreover, using also the last item of
Proposition \ref{regolo} 
we may express 
\be\label{YMU}
\begin{aligned}
& Y = (\mathfrak{C}^{\theta}(U)
\circ\mathfrak{F}^{\theta}(U))_{|_{\theta=1}}[U] =U+\widetilde{\mathtt{M}}(U)[U] \,, \\
& \widetilde{\mathtt{M}}(U) \in  \Sigma\mathcal{M}^{m_2}_{K,K',1}[r,2]\otimes\mathcal{M}_{2}(\C) \, , \ m_2\geq 3/ 2 \, . 
\end{aligned}
\ee  
Then system 
\eqref{BNF12} can be written as system 
\eqref{finalsyst1012} with $\mathcal{X}_{\geq 3}(U,Y)$ given in \eqref{Stimaenergy100} and
$$
{\mathfrak R}_{\geq 2} (U)  := 
  \mathtt{R}_1^{res} (U) - \mathtt{R}_1^{res} (U+\widetilde{\mathtt{M}}(U)[U]) 
  + {\mathtt{R}}_{\geq2}(U) \, .
$$
By \eqref{YMU} and Proposition \ref{composizioniTOTALI}-$(iii)$
we have that
 ${\mathfrak R}_{\geq 2} (U)\in
\Sigma\mathcal{R}^{- (\rho - \rho_0)}_{K,K',2}\otimes\mathcal{M}_{2}(\C)$
where  $ \rho_0 := \max \{ m_1, m_2 \} $.
\end{proof}

\section{Birkhoff normal form and quadratic life-span of solutions}

In this section we prove  Theorems \ref{BNFtheorem} and \ref{thm:main2}.
We first recall the Hamiltonian formalism in the complex symplectic variables
\begin{equation}\label{complexCoord}
\left(\begin{matrix}
u \\ \bar{u}
\end{matrix}\right):=\mathtt{B} \left(\begin{matrix}
\eta \\ \psi
\end{matrix}\right)=
\frac{1}{\sqrt{2}}\left(\begin{matrix}
 \Lambda\psi+\ii \Lambda^{-1}\eta \\
 \Lambda\psi-\ii\Lambda^{-1}\eta
\end{matrix}
\right), 
\left(\begin{matrix}
\eta \\ \psi
\end{matrix}\right):=\mathtt{B}^{-1} \left(\begin{matrix}
u \\ \bar{u}
\end{matrix}\right)=\frac{1}{ \sqrt{2}}
\left(\begin{matrix}
-\ii \Lambda (u-\bar{u}) \\
\Lambda^{-1}(u+\bar{u})
\end{matrix}
\right),
\end{equation}
where $\Lambda$ is the Fourier multiplier defined in \eqref{compl1}.

A vector field $X(\eta,\psi)$ and a function $H(\eta,\psi)$ assume
 the form
\begin{equation}\label{vec:compl}
X^{\mathbb{C}}:=\mathtt{B}^{\star}X:=\mathtt{B}X\mathtt{B}^{-1}\,,
\qquad H_{\mathbb{C}}:=H\circ\mathtt{B}^{-1}\,.
\end{equation}
The Poisson bracket in \eqref{PoiBra}  reads
$ \{F_{\mathbb{C}},H_{\mathbb{C}}\}:=\ii
\sum_{j\in \mathbb{Z}\setminus\{0\}}\pa_{u_{j}}H_{\mathbb{C}}\pa_{\ov{u_{j}}}F_{\mathbb{C}}
-\pa_{\ov{u_{j}}}H_{\mathbb{C}}\pa_{{u_{j}}}F_{\mathbb{C}} $. 

Given a Hamiltonian $F_{\mathbb{C}}$, expressed in the complex
variables $(u,\bar{u})$,
the associated Hamiltonian 
vector field 
$ X_{F_{\C}} $ is 
\be\label{HamVecField}
X_{F_{\C}} =
 \left(
\begin{matrix}
\ii\pa_{\bar{u}}F_{\mathbb{C}} \\
-\ii\pa_{u}F_{\mathbb{C}}
\end{matrix}
\right)
= \frac{1}{\sqrt{2 \pi}} \sum_{k\in\Z\setminus\{0\}} 
  \left(\begin{matrix}
 \ii  \pa_{ \ov{u_k}} F_{\C} \, e^{\ii k x}  \\
- \ii  \pa_{u_k} F_{\C}   \, e^{- \ii k x} 
\end{matrix} 
\right) \, ,
\ee
that we also identify, using the  standard vector field notation, with 
$$X_{F_{\C}} 
=\sum_{k\in\Z\setminus\{0\}, \sigma=\pm} \ii \s \pa_{u^{-\s}_{k}} F_\C  \, \pa_{u^{\s}_{k}} \, .
$$
If  $X_{F}$ is the Hamiltonian vector field 
 of the Hamiltonian $F:=F_{\mathbb{C}}\circ\mathtt{B}$, we have 
\begin{equation}\label{compreal}
X_{F}^{\C} := \mathtt{B}^\star X_F =  X_{F_{\C}}  \, . 
\end{equation}
The push-forward acts naturally on the commutator of nonlinear vector fields, 
defined in \eqref{nonlinCommu},  namely 
\be\label{PushF}
B^\star \bral X, Y\brar =  \bral B^\star X, B^\star Y\brar = \bral X^\C, Y^\C \brar  \, . 
\ee
Recalling  \eqref{Hamvera},  the Hamiltonian  \eqref{Hamiltonian}  admits,  
in complex coordinates,  the expansion 
$$
H_{\mathbb{C}}:=
H\circ \mathtt{B}^{-1}
=H^{(2)}_{\mathbb{C}}
+H^{(3)}_{\mathbb{C}} + \ldots 
$$
where, recalling \eqref{compVar}, \eqref{dispersionLaw}, \eqref{complex-uU}, 
\begin{equation}\label{Hamexp2}
\begin{aligned}
&H^{(2)}_{\mathbb{C}}=\sum_{j\in \mathbb{Z}\setminus\{0\}}
\mk(j)u_{j} \ov{u_j}\,,\qquad H^{(3)}_{\mathbb{C}}=
\sum_{\s_1 j_1 + \s_2 j_2 + \s_3 j_3 = 0 } H_{j_1, j_2, j_3}^{\s_1, \s_2, \s_3} u_{j_1}^{\s_1} u_{j_2}^{\s_2} u_{j_3}^{\s_3} 
\end{aligned}
\end{equation}
and  $H_{j_1, j_2, j_3}^{\s_1, \s_2, \s_3}  $ are computed in \eqref{coeffH3},
for $ j_1, j_2, j_3 \in \Z \setminus \{0\} $.

\subsection{Normal form identification and proof of Theorem \ref{BNFtheorem}} 
\label{sec:IDE}

A normal form uniqueness argument 
allows to 
identify the quadratic Poincar\'e-Birkhoff 
resonant vector field $\mathtt{R}_1^{res}(Y)[Y] $ in \eqref{finalsyst1012}
as the cubic resonant Hamiltonian vector field obtained by the 
formal Birkhoff normal form construction in \cite{CS}.

\begin{proposition} {{\bf (Identification of the quadratic resonant Birkhoff 
normal form)}}\label{lem:ide3} 
The Birkhoff resonant vector field 
$ \mathtt{R}_1^{res}(Y)[Y]  $ defined in \eqref{finalsyst1012} is equal to 
\begin{equation}\label{lem:identi}
\mathtt{R}_1^{res}(Y)[Y]= 
X_{H_{BNF}^{(3)}}
\end{equation}
where $ H_{BNF}^{(3)} $ is the cubic Birkhoff normal form Hamiltonian in \eqref{H3bnf}. 
\end{proposition}

The proof follows the ideas developed in Section $7$ in \cite{BFP}. 
Recalling \eqref{Hamvera}, we first expand  the water waves Hamiltonian vector field in  
\eqref{WW1}-\eqref{Hamiltonian}
in degrees of homogeneity
\be\label{HVF2Sez4}
X_H =  X_1 + X_2 +  X_{\geq 3}   \qquad {\rm where} \qquad 
X_{1}:=X_{H^{(2)}}, \  X_{2}:=X_{H^{(3)}} \, ,
\ee
and $X_{\geq 3}  $ collects the higher order terms. 
System \eqref{finalsyst1012} has been obtained 
conjugating 
\eqref{WW1} 
under the map 
\begin{equation}\label{var:finSez4}
Y = {\bf F}^{1}(U) \circ \mathtt{B}\circ\mathcal{G}\vect{\eta}{\psi} \, ,
\end{equation}
where  $ {\mathcal G} $ is 
the good-unknown transformation (see \eqref{omega0})
\be\label{Alinach-good} 
\vect{\eta}{\omega} =
{\mathcal G} \vect{\eta}{\psi} := 
 \vect{\eta}{ \psi - \opbw{(B(\eta, \psi)) \eta}  } \, ,
\ee
the map $ \mathtt{B} $ is defined in \eqref{complexCoord}
(see \eqref{compVar}  and Proposition \ref{Formulazione}), 
and 
\begin{equation}\label{mappaTotaSez4}
{\bf F}^{\theta}(U):=\mathfrak{C}^{\theta}(U)\circ \mathfrak{F}^{\theta}(U)\, , 
\;\; \theta\in [0,1]\,,
\end{equation}
where $\mathfrak{F}^{\theta}(U)$, $\mathfrak{C}^{\theta}(U)$ 
are defined respectively in  Propositions \ref{regolo} and \ref{cor:BNF10}.
In order to identify the quadratic 
vector field in system \eqref{finalsyst1012},  
we perform a Lie commutator expansion,  
up to terms of homogeneity at least $3$. 
Notice that the quadratic  term in \eqref{finalsyst1012} may arise by only the conjugation of
 $ X_1 + X_2  $ under the
homogeneous components of the paradifferential transformations $ {\mathcal G} $ 
and $ {\bf F}^{1}(U)$, neglecting cubic terms. 

We  use the following 
Lemma \ref{Lieexp-Inve} 
that collects Lemmata $A.8$, $A.9$ and $A.10$ in \cite{BFP}.
The variable $U$ may denote both the 
couple of complex variables $(u,\bar{u})$ 
or the real variables  $(\eta,\psi)$.

\begin{lemma}[{ \cite{BFP}}] {\bf (Lie expansion)}\label{Lieexp-Inve}
Consider  a map $ \theta \mapsto {\bf F}_{\leq 2}^{\theta}(U) $, 
$ \theta \in [0,1] $, 
of the form 
\begin{equation}\label{espmultilin}
{\bf F}^{\theta}_{\leq 2} (U)=U
+\theta M_{1}(U)[U]\,, \qquad
 M_{1}(U) \in \widetilde{{\mathcal{M}}}_{1}\otimes\mathcal{M}_2(\C)\,.
\end{equation} 
Then:

\vspace{0.5em}
\noindent
$(i)$ the family of maps ${\bf G}_{\leq 2}^{\theta}(V):=V- \theta  M_{1}(V)[V] $
is such that
$$
{\bf G}_{\leq 2}^{\theta}\circ{\bf F}_{\leq 2}^{\theta}(U)= U + M_{\geq 2}(\theta; U)[U] \, ,  \quad
{\bf F}_{\leq 2}^{\theta}\circ{\bf G}_{\leq 2}^{\theta}(V)= V + M_{\geq 2}(\theta; U)[U] \, , 
$$
where $ M_{\geq 2}(\theta; U) $ is a polynomial in 
$ \theta $ and finitely many monomials  
$ M_p (U)[U] $ for $ M_p (U) \in  \widetilde{{\mathcal{M}}}_{p} \otimes\mathcal{M}_2(\C) $, $ p \geq 2 $;

\vspace{0.5em}
\noindent
$(ii)$ the family of maps ${\bf G}_{\leq 2}^\theta (V) $
satisfies
$$
\pa_{\theta}{\bf G}_{\leq 2}^{\theta}(V) 
= S({\bf G}_{\leq 2}^{\theta}(V))+M_{\geq 2}(\theta;U)[U] \, , \quad 
{\bf G}_{\leq 2}^{0}(V) = V \, , 
$$
where 
$ S(U)=  S_1(U)[U]$ with 
 $ S_{1}(U)\in \widetilde{\mathcal{M}}_{1}\otimes\mathcal{M}_2(\C) $ 
 and 
$ M_{\geq 2}(\theta; U) $ is a polynomial in $ \theta $ and finitely many monomials  
$ M_p (U)[U] $ for maps $ M_p (U) \in  \widetilde{{\mathcal{M}}}_{p} \otimes\mathcal{M}_2(\C) $, $ p \geq 2 $. 

\vspace{0.5em}
\noindent
(iii) 
Let $ X(U) = M(U)U $ for some map $ M (U) = M_0 + M_1 (U)  $ 
where  $ M_0 $ is in $ \widetilde{\mathcal M}_0 \otimes {\mathcal M}_2 (\C) $
and 
$ M_1 (U) $ in $ \widetilde{\mathcal M}_1 \otimes {\mathcal M}_2 (\C) $.
If $U$ solves $\pa_t U =X(U)$, 
then 
the function  $V := {\bf F}^{1}_{\leq 2} (U) $ solves
\be\label{expansione appro2}
\pa_t V =  X(V) + \bral S, X\brar(V) +  \cdots \, , 
\ee
up to terms of degree of homogeneity greater or equal to $3$, 
where we define the nonlinear commutator 
\be\label{nonlinCommu}
\bral S, X\brar (U)  := d_U X (U) [S(U)] -  d_U S (U) [X(U)] \, . 
\ee
\end{lemma}

\begin{itemize}
\item {\bf Notation.} 
Given a homogeneous vector field $ X $, we denote
by $ \Phi_{S}^\star X $ the induced (formal) push forward (see \eqref{expansione appro2})  
\be\label{push-forward}
\Phi_{S}^\star X = X + \bral S,X \brar+ \cdots 
\ee
where the dots  $ \cdots $ 
denote cubic terms. 
\end{itemize}

\noindent
{\bf Proof of Proposition \ref{lem:ide3}.}
\\[1mm]
{\bf Step $ 1$. The good unknown change of variable $ {\mathcal G} $ in \eqref{Alinach-good}}. 
First of all we 
note that $ {\mathcal G}(\eta,\psi)=(\Phi^{\theta}(\eta, \psi))_{\theta=1} $ 
where
$$
\Phi^{\theta} \vect{\eta}{\psi} =
 \vect{\eta}{ \psi - \theta \opbw{(B(\eta, \psi)) \eta}  } \, ,\quad \theta\in [0,1] \, . 
$$
Since $ B(\eta, \psi) $ is a function in 
$ \Sigma {\mathcal F}^\R_{K,0,1}[r,2] $ we have that 
$\Phi^{\theta}(\eta,\psi)$  has an expansion as in \eqref{espmultilin} 
up to cubic terms.
Hence, by Lemma \ref{Lieexp-Inve}-$(i)$-$(ii)$, 
we regard the inverse of the map $ {\mathcal G}_{\leq 2}$,  obtained 
truncating $ {\cal G} $ up to cubic remainders, 
as  the (formal) 
 time one flow of a  quadratic 
 vector field 
\be\label{S34Sez4}
\mathtt{S}_2 := S_1(\eta,\psi)\vect{\eta}{\psi} \, , 
\quad 
S_{1}(\eta,\psi) \in \widetilde{\mathcal{M}}_{1}\otimes\mathcal{M}_2(\C)\,.
\ee
By \eqref{HVF2Sez4}, \eqref{push-forward} and  \eqref{S34Sez4}, we get
\begin{align}
\Phi_{\mathtt{S}_2}^\star (X_1+X_2) &  = X_1 + 
X_2 + \bral\mathtt{S}_2, X_1\brar  
+ \cdots \, . \label{VFB}
\end{align}
\noindent
{\bf Step $ 2$.  Complex coordinates $\mathtt{B}$ in \eqref{complexCoord}}. 
In the complex 
coordinates \eqref{complexCoord},  the vector field 
\eqref{VFB} 
reads, recalling  \eqref{vec:compl} and \eqref{PushF}, 
\be\label{VFComplex}
\mathtt{B}^\star \Phi_{\mathtt{S}_2}^\star (X_1+X_2) = 
X_1^\C +  X_2^\C + \bral\mathtt{S}_2^\C, X_1^\C\brar 
+ \cdots 
\ee
where, by \eqref{compreal}, \eqref{HVF2Sez4}, \eqref{Hamexp2},
\be\label{X1X2C}
X_{1}^{\C}  = X_{H_\C^{(2)}} =   \ii  \sum_{j, \s} \sigma \mk(j) u_j^\s \pa_{u_j^\s} \, , \quad
X_2^\C  = X_{H^{(3)}_{\mathbb{C}}} \, .
\ee
{\bf  Step $ 3 $. The transformation $ {\bf F}^{1}$  in \eqref{mappaTotaSez4}.}
By the last items of Proposition \ref{regolo} and Proposition \ref{cor:BNF10}, 
the map ${\bf F}^{\theta}(U) $ 
has the form  \eqref{espmultilin} 
up to cubic terms. Thus, by Lemma \ref{Lieexp-Inve}-$(i)$-$(ii)$,
 the approximate inverse of the truncated 
 map $ {\bf F}_{\leq 2} ^1 $ 
can be regarded  as the (formal) 
 time-one flow of a  vector field 
\begin{equation}\label{T34}
\mathtt{T}_2:=T_1(U)[U] \, , \quad   
  T_{1}(U)\in \widetilde{\mathcal{M}}_{1}\otimes\mathcal{M}_2(\C) \,.
\end{equation}
By \eqref{VFComplex}, \eqref{X1X2C}, \eqref{push-forward}, we get
\begin{equation}\label{Liegrad}
\Phi_{\mathtt{T}_2}^\star \mathtt{B}^\star \Phi_{\mathtt{S}_2}^\star (X_1+X_2)  =  
X_{H^{(2)}_\C} +  X_{H^{(3)}_{\mathbb{C}}}  
+ \bral  \mathtt{S}_2^\C + \mathtt{T}_2, X_{H^{(2)}_\C} \brar  + \cdots \, . 
\end{equation}
Comparing \eqref{finalsyst1012} and \eqref{Liegrad} we deduce that 
\begin{equation}\label{condizioniSez4}
\mathtt{R}_1^{res}(Y)[Y] \equiv 
X_{H^{(3)}_{\mathbb{C}}}    + \bral  \mathtt{S}_2^\C + \mathtt{T}_2, X_{H^{(2)}_\C} \brar \,.
\end{equation}
The vector field $ \mathtt{R}_1^{res}(Y)[Y] $ is in Poincar\'e-Birkhoff normal form,
recall  Definition \ref{Resonant}. 
Therefore, defining the linear operator 
$ \Pi_{\ker}  $ acting on a quadratic monomial vector field
$ u_{j_1}^{\s_1} u_{j_2}^{\s_2} \pa_{u^{\s}_j}$  
as
\be\label{PikerdiSez4}
\Pi_{\ker} \Big( u_{j_1}^{\s_1} u_{j_2}^{\s_2}  \pa_{u^{\s}_j} \Big) := 
\begin{cases}
u_{j_1}^{\s_1} u_{j_2}^{\s_2}  \pa_{u^{\s}_j} \quad \quad 
{\rm if} \ - \s \mk(j) + \s_1 \mk({j_1}) + \s_2 \mk({j_2}) = 0 \\
0 \qquad \qquad\quad \quad {\rm otherwise} \, ,
\end{cases}
\ee
we have that
\be\label{kerRR}
\mathtt{R}_1^{res}(Y)[Y] = \Pi_{\ker}(\mathtt{R}_1^{res}(Y)[Y])    \, . 
\ee
In addition, since 
\[
 \bral  u_{j_1}^{\s_1} u_{j_2}^{\s_2}  \pa_{u^{\s}_j} , X_{H^{(2)}_\C} \brar = 
 \ii \big(\s \mk(j) - \s_1 \mk({j_1}) - \s_2 \mk({j_2})  \big) 
u_{j_1}^{\s_1} u_{j_2}^{\s_2}    \pa_{u^{\s}_j} \,,
\]
we deduce
\begin{equation}\label{sononelKerSez4}
\Pi_{{\rm ker}} \bral \mathtt{S}_2^\C 
+ \mathtt{T}_2 , X_{H^{(2)}_\C} \brar =  0\,.
\end{equation}
In conclusion,  \eqref{kerRR}, \eqref{condizioniSez4} and \eqref{sononelKerSez4}
imply that
\[
\mathtt{R}_1^{res}(Y)[Y] = \Pi_{{\rm ker}}(X_{H^{(3)}_{\mathbb{C}}})
\stackrel{\eqref{Hamexp2} } = X_{H_{BNF}^{(3)}}
\]
where $ H_{BNF}^{(3)} $ is the Hamiltonian in \eqref{H3bnf}. 
This proves \eqref{lem:identi}.

\vspace{0.5em}
\noindent
{\bf Proof of Theorem \ref{BNFtheorem}. }
Hypothesis \eqref{hypoeta} implies that 
 the variable $u$  defined  in \eqref{compVar} satisfies 
\eqref{pallaU}
and therefore the function $U=\vect{u}{\bar{u}}$
belongs to the ball  $B_{s}^{K}(I;r)$  (recall \eqref{palla}) 
 with $ r = C_{s,K} \bar{\e} \ll 1$ and $I=[-T,T]$. 
By Proposition \ref{Formulazione} the function
 $U$ solves system \eqref{complEQ}.
 Then we apply Proposition \ref{regolo} and  the Poincar\'e-Birkhoff Proposition \ref{cor:BNF10}  with $s\gg K\geq 
K'(\rho)$ and $K'(\rho)$ given by Proposition \ref{regolo}, taking  
$ \bar \e $ small enough.
The map ${\bf F}^{1}(U) $ in \eqref{mappaTotaSez4} 
transforms the water waves system  \eqref{complEQ} 
into \eqref{finalsyst1012}, 
which, thanks to  Proposition \ref{lem:ide3}, 
is expressed in terms of the  Hamiltonian $H_{BNF}^{(3)} $ 
in \eqref{H3bnf} as
\begin{equation*}
 \pa_{t}y=\ii \mk(D)y +\ii \pa_{\bar{y}}H_{BNF}^{(3)}(y,\bar{y}) 
+ {\mathcal X}^{+}_{\geq 3} \, 
\end{equation*}
where $\mathcal{X}_{\geq 3}^+$ is the first component of ${\mathcal X}_{\geq3}(U,Y)$ in \eqref{Stimaenergy100}. 
Renaming $ y \rightsquigarrow z $,
the above equation is \eqref{theoBireq}.
We define
 $ z = \mathfrak{B}(\eta, \psi) [\eta, \psi] $  as 
the first component of the change of variable \eqref{var:finSez4}, namely   
of ${\bf F}^{1}(U) \circ {\mathtt B} \circ {\cal G} [\eta, \psi] $, with
$ U $ written in terms of $ (\eta, \psi) $ by \eqref{compVar},  \eqref{omega0}. 
By \eqref{stimafinaleFINFRAK}  and \eqref{stimFINFRAK} with $k=0$, 
and using that $U \in B_{s}^{K}(I;r)$,  we get 
\begin{equation}\label{secondaequiv}
\| z (t) \|_{\dot{H}^{s}}\sim_{s}\|u (t) \|_{\dot{H}^{s}} \,,  
\end{equation} 
and  \eqref{Germe} follows, using also   \eqref{compVar}, \eqref{omega0}.
The cubic vector field $ \mathcal{X}_{\geq3}^{+} $  in \eqref{Stimaenergy100} 
satisfies the estimate
$ \|\mathcal{X}_{\geq3}^{+}\|_{\dot{H}^{s-\frac{3}{2}}}\lesssim_s \|z\|^{3}_{\dot{H}^{s}} $
by Proposition 3.8 in \cite{BD} (recall that 
$ \mathcal{H}_{\geq2}(U;\x) 
\in \Gamma^{3/2}_{K,K',2}[r]\otimes\mathcal{M}_2(\mathbb{C}) $), 
by  \eqref{piove} with $ k = 0 $, and \eqref{lem:tempo}, \eqref{secondaequiv}.
Moreover, the vector field $\mathcal{X}_{\geq3}^{+}$ 
satisfies
the energy estimate \eqref{theoBirR}
since the symbol ${\mathcal{H}}_{\geq2}(U;  \xi)$ 
is independent of $x$ and purely imaginary up to symbols of order $ 0 $,   
see \eqref{quasi-real} (for the detailed argument we refer to Lemma $7.5$ in \cite{BFP}).

\subsection{Energy estimate and proof of Theorem \ref{thm:main2}}\label{sec:EE}

We now deduce Theorem \ref{thm:main2} by Theorem \ref{BNFtheorem} and
the following energy estimate for the solution $ z $ of the Birkhoff 
resonant system \eqref{theoBireq}. 
By time reversibility,  without loss of generality, we may only look at positive times $ t > 0 $.

\begin{lemma} {\bf (Energy estimate)}\label{lem:SE}
Fix $ s , \bar{\e} > 0 $ as in Theorem  \ref{BNFtheorem}  and 
assume that 
the solution $(\eta,\psi)$  of \eqref{WW1} satifies \eqref{hypoeta}.
Then the solution  $ z (t) $ of \eqref{theoBireq} satisfies 
\begin{equation}\label{stimaconclu}
\| z (t) \|_{{\dot H}^s}^2 \leq C(s) \| z (0) \|_{{\dot H}^s}^2 + C(s) \int_0^t 
\| z(\tau)\|_{{\dot H}^s}^4 \, d \tau \, , \quad \forall  t \in [0, T ]    \,   .
\end{equation}
\end{lemma}

\begin{proof}
By Lemma  \ref{stimedalbasso}, the Birkhoff resonant Hamiltonian $ H^{(3)}_{BNF} $ in 
\eqref{H3bnf} depends on finitely many variables $ z_{j_1}^\pm,  z_{j_2}^\pm,  z_{j_3}^\pm $, 
$ j_1, j_2, j_3 \in \Z \setminus \{0\} $,  because 
\be
\begin{cases} \label{maxj1j2j3} 
\s_1 j_1 + \s_2 j_2 + \s_3 j_3 = 0 \\
 \s_1 \mk(j_1) + \s_2 \mk(j_2) + \s_3 \mk(j_3) = 0
 \end{cases}
 \quad \Longrightarrow  \quad \max(|j_1|, |j_2|, |j_3|) < {\mathtt C} \, .
\ee
 For any function 
$ w \in {\dot H}^s (\T) $ 
we define the projector $\Pi_{L}$ on low modes, respectively the projector 
$\Pi_{H}$ on high modes, as
$$
w_{L}:=\Pi_{L}w:=\frac{1}{\sqrt{2\pi}}
\sum_{0< |j|\leq \mathtt{C}}w_{j}e^{\ii jx}\,,\qquad
w_{H}:=\Pi_{H}w:=\frac{1}{\sqrt{2\pi}}\sum_{|j|> \mathtt{C}}w_{j}e^{\ii jx} \, . 
$$
We write $w=w_{L}+w_{H}$ and we define the norm
$$
\|w\|_{s}^{2}:=H^{(2)}_{\mathbb{C}}(w_{L})+\|w_{H}\|^2_{\dot{H}^{s}}
$$
where (see \eqref{Hamexp2}) 
\be\label{H2-comp}
H^{(2)}_{\mathbb{C}} (w) = 
\int_{\T} \mk (D)  w \cdot\bar{w} \, dx = \sum_{j \in \Z \setminus \{0\}}\mk (j) w_j \ov{w_j}\, .  
\ee
Since $ \Omega (j) > 0 $, $ \forall j \neq 0 $,  and $ w_L $ is supported on finitely many Fourier modes
$ 0 < |j | \leq {\mathtt C} $, 
we have that, for some constant $ C_s > 0 $,  
\begin{equation}\label{normaequiv2}
C_{s}^{-1}\|w\|_{s}\leq \|w\|_{\dot{H}^{s}}\leq C_{s}\|w\|_{s}  \, , 
\end{equation}
i.e. the norms $\|\cdot\|_{s}$ and $\|\cdot\|_{\dot{H}^{s}}$ are equivalent.
We now prove the estimate \eqref{stimaconclu} for the 
equivalent norm $\|\cdot\|_{s}$.

We first note that, by \eqref{maxj1j2j3}, $   H^{(3)}_{BNF}  (z, \bar z)  = H^{(3)}_{BNF} 
 (z_L, \bar z_L )  $.
Therefore 
$ \Pi_H \pa_{\bar{z}} H^{(3)}_{BNF}  (z, \bar z) = 0 $ and 
 the equation \eqref{theoBireq} amounts to the system  
\be\label{sistema-BNF}
\begin{cases}
\dot z_{L}  = \ii \mk(D) z_{L} + \ii \pa_{\bar{z}} H^{(3)}_{BNF}  (z_{L}, \bar z_L)
+\Pi_{L}\big( \mathcal{X}^{+}_{\geq3}(U,Z)\big) \cr
\dot{z}_{H} =\ii \mk(D) z_{H}
\qquad\quad\qquad \qquad  \quad \   +\Pi_{H}\big( \mathcal{X}^{+}_{\geq3}(U,Z)\big) \, . 
\end{cases}
\ee
Moreover since the Hamiltonian $ H^{(3)}_{BNF}  $ in \eqref{H3bnf} is in Birkhoff normal form, it Poisson commutes with the quadratic Hamiltonian $ H^{(2)}_{\mathbb{C}} $
in \eqref{H2-comp}, i.e. 
 \begin{equation}\label{Zcommu}
  \{H^{(3)}_{BNF}, H^{(2)}_{\mathbb{C}} \} =  0 \, . 
  \end{equation}
We have
\begin{align}\label{secondastima}
\pa_{t}H^{(2)}_{\mathbb{C}}(z_{L})&  \stackrel{\eqref{sistema-BNF}}  =
\{H^{(3)}_{BNF} , H^{(2)}_{\mathbb{C}}\} +
2 {\rm Re}\int_{\mathbb{T}}\mk(D) \Pi_{L}\big(
\mathcal{X}_{\geq3}^{+}(U,Z)\big)\cdot \bar{z}_{L}dx \nonumber \\
& \stackrel{\eqref{Zcommu}} = 2 {\rm Re} 
\int_{\mathbb{T}} \Pi_{L} \mk(D) \big(
\mathcal{X}_{\geq3}^{+}(U,Z)\big)\cdot \Pi_{L} \bar z \, dx 
 \lesssim_{s} \| z \|^{4}_{\dot{H}^{s}}
\end{align}
using  
that $\|\Pi_{L} \mk(D) \mathcal{X}_{\geq3}^{+}\|_{\dot{H}^{0}}
\lesssim_{s}\|z\|_{\dot{H}^{s}}^{3} $
by item $(2)$ of Theorem \ref{BNFtheorem}.
Moreover,  since $ \Pi_H  $ and $ \Pi_L  $ project on  $ L^2 $-orthogonal subspaces, 
\begin{align}\label{primastima}
\pa_{t}\|z_{H}\|^{2}_{\dot{H}^{s}} & =\pa_{t}(|D|^s z_{H}, |D|^{s}z_{H})_{L^{2}} \nonumber  \stackrel{\eqref{sistema-BNF}} = 2{\rm Re}\int_{\mathbb{T}}|D|^{s}\Pi_{H}\big(
\mathcal{X}_{\geq3}^{+}(U,Z)\big)\cdot |D|^{s}  \Pi_H \ov{z} \, dx 
\\ &=
2{\rm Re}\int_{\mathbb{T}}|D|^{s}
\mathcal{X}_{\geq3}^{+}(U,Z)\cdot |D|^{s}  \ov{z} \, dx 
-
2{\rm Re}\int_{\mathbb{T}}|D|^{s}\Pi_{L}\big(
\mathcal{X}_{\geq3}^{+}(U,Z)\big)\cdot |D|^{s}  \Pi_L \ov{z} \, dx \nonumber
\\&
 \stackrel{\eqref{theoBirR}}{\lesssim_s}  
\| z \|^{4}_{\dot{H}^{s}}+\|\Pi_{L}\mathcal{X}_{\geq3}^{+}\|_{\dot{H}^{s}}
\|\Pi_{L}z\|_{\dot{H}^{s}}\lesssim_{s}\| z \|^{4}_{\dot{H}^{s}}
\end{align}
by item $(2)$ of Theorem \ref{BNFtheorem}.
Integrating in $ t $ the inequalities 
\eqref{secondastima}, \eqref{primastima}, we deduce
\[
\| z(t)\|^{2}_{s}\lesssim_{s}\|z(0)\|^{2}_{s}
+\int_{0}^{t}\|z(\tau)\|^{4}_{\dot{H}^{s}}d\tau
\]
which, together with the equivalence \eqref{normaequiv2}, implies  \eqref{stimaconclu}.
\end{proof}

\begin{proof}[{\bf Conclusion of the Proof of Theorem \ref{thm:main2}}]
 Consider  initial data $(\eta_0,\psi_0)$ satisfying \eqref{smallness}
 with $s\gg 1$ given by Theorem \ref{BNFtheorem}.
Classical local existence results imply that 
$$
(\eta,\psi)\in C^{0}\big([0,T_{\rm loc}], 
H_0^{s+\frac{1}{4}}(\mathbb{T},\mathbb{R})\times
  \dot{H}^{s-\frac{1}{4}}(\mathbb{T},\mathbb{R})\big)
$$
for some $ T_{\rm loc} > 0 $ and thus \eqref{hypoeta} holds with $ \bar \e = 2 \e $
and  $  T = T_{\rm loc} $. 
A standard bootstrap argument based 
on the energy estimate \eqref{stimaconclu} 
(see for instance Proposition $7.6$ in \cite{BFP}) implies  that the solution
$z(t)$ of \eqref{theoBireq} can be extended up to a time 
$ T_{\e} := c_0 \e^{-2}$ for some $c_0>0$, and satisfies
 \begin{equation}\label{pallaU3sez4}
\sup_{t\in [0,T_{\e}]} \| z(t) \|_{\dot{H}^{s}}\lesssim_{s} \e\,.
\end{equation}
We deduce \eqref{smallness2} by 
\eqref{pallaU3sez4}, the equivalence \eqref{secondaequiv}, and 
going back to the original variables $ (\eta, \psi ) $ by \eqref{compVar} and 
\eqref{omega0}.
\end{proof}

{\sc Massimiliano Berti, SISSA}, Trieste,  berti@sissa.it

{\sc Roberto Feola, University of Nantes}, roberto.feola@univ-nantes.fr

{\sc Luca Franzoi, SISSA},  Trieste, luca.franzoi@sissa.it

\end{document}